\documentclass{article}
\usepackage{graphicx} 

\usepackage{amssymb,amsmath,amsfonts,amsthm}
\usepackage{latexsym}
\usepackage[abbrev,msc-links,nobysame]{amsrefs}
\usepackage{mathtools}
\usepackage{hyperref}
\usepackage{color}

\newtheorem{theorem}{Theorem}[section]

\newtheorem{lemma}[theorem]{Lemma}
\newtheorem{corollary}[theorem]{Corollary}
\newtheorem{proposition}[theorem]{Proposition}

\newtheorem*{claim*}{Claim}

\newtheorem*{conjecture*}{Conjecture}

\theoremstyle{definition}
\newtheorem{definition}[theorem]{Definition}
\theoremstyle{remark}   
\newtheorem{remark}[theorem]{Remark}
\newtheorem*{remark*}{Remark}
\newtheorem{example}[theorem]{Example}

\theoremstyle{remark}

\DeclareMathOperator{\supp}{supp}
\DeclareMathOperator{\dvg}{div}
\DeclareMathOperator{\diam}{diam}
\DeclareMathOperator{\Hess}{Hess}
\DeclareMathOperator{\Ric}{Ric}

\newcommand{\R}{\mathbb{R}}

\title{Bakry-\'Emery calculus for entropic curvature, new diameter estimates, and spectral gaps}

\author{
  Supanat Kamtue\footnote{Yau Mathematical Sciences Center, Tsinghua University, Beijing, China, \\
  skamtue@tsinghua.edu.cn}
  \and Shiping Liu\footnote{School of Mathematical Sciences, University of Science and Technology of China, Hefei 230026, China, 
  spliu@ustc.edu.cn}
  \and Florentin M\"unch\footnote{MPI MiS Leipzig, 04103 Leipzig, Germany,  florentin.muench@mis.mpg.de}
  \and Norbert Peyerimhoff\footnote{Department of Mathematical Sciences, Durham University, Durham, UK, \\norbert.peyerimhoff@durham.ac.uk}
}

\date{\today}

\begin{document}

\maketitle

\begin{abstract}
In this paper, we propose a generalization of Bakry-\'Emery's calculus which allows us to formulate both Bakry-\'Emery and entropic curvature simultaneously. This formulation represents both curvatures as an integral of the Bochner formula against some measure. This leads to a natural optimality criterion of measures, which we investigate in the Bakry-\'Emery setting. Moreover, our approach leads also to a dimension parameter in the framework of entropic curvature. 

We also present gradient estimate applications, that is, diameter estimates for Markov chains with strictly positive entropic curvature and a spectral gap estimate. The latter implies non-existence of non-negatively curved expanders.
\end{abstract}

\tableofcontents

\section{Introduction}

Erbar and Maas \cite{EM-12} introduced in 2012 \emph{entropic Ricci curvature} for finite Markov chains (see \cite{Mi-13} for independent work in the same direction, and see also \cite{EFS-19} for its variant adapted to non-linear Markov chains). Several applications to study entropic curvature of interacting particle systems have been discussed in \cite{FM-16,EMT-15,EHMT-17}.
This curvature was motivated by the Ricci curvature notion introduced by Sturm \cite{St-06-i,St-06-ii} and Lott--Villani \cite{LV-09} for metric measure spaces, which is based on the convexity of the entropy functional along $W_2$-Wasserstein geodesics.  A main obstacle in the discrete setting is that 
there are no non-constant geodesics in the $W_2$-Wasserstein metric. To circumvent this problem, the Wasserstein metric is modified involving a logarithmic mean $\theta_{log}$ (see \cite{Ma-11} and also \cite{Mi-13,CHLZ-12}), which helps to accommodate the lack of chain rule $\nabla \log\rho = \frac{1}{\rho} \nabla \rho$ in the discrete setting. More importantly, it helps to establish the discrete analogue of the Jordan-Kinderlehrer-Otto (JKO) result, which characterizes the heat flow as the Wasserstein-gradient flow of the entropy \cite{JKO-98}. In fact, it was shown in \cite{GM-13} that this modified Wasserstein metric on the discrete torus converges in the sense of Gromov-Hausdorff to the continuous counterpart. In analogy to the definition by Sturm and Lott--Villani, entropic curvature is a global invariant describing convexity of the entropy along geodesics with respect to this modified Wasserstein metric.  There is also an equivalent infinitesimal description of entropic curvature via the Hessian of entropy, leading to an integrated version of a Bochner-type inequality. The original formula by Bochner is a pointwise identity in the setting of Riemannian manifolds, involving the Laplacian, the Hessian and Ricci curvature. 

Another widely used Ricci curvature notion for discrete spaces, related again to Bochner's identity, is 
\emph{Bakry-\'Emery curvature} (see \cite{BE-85,Elw-91,Schm-98,LY-10}). 
This curvature requires a (not necessarily integer-valued) dimension parameter $n \in (0,\infty]$ and is defined on the vertices of a weighted graph. Bakry-\'Emery curvature emerges via a derivation of Bochner's formula using bilinear forms $\Gamma$ and $\Gamma_2$ involving only the Laplacian (referred to as $\Gamma$ calculus), leading to the so-called curvature-dimension inequality $CD(K,n)$. In Section \ref{sec:overview}, we present a more detailed introduction into both curvature notions for the readers' convenience.

Other well-known discrete Ricci curvature notions are Ollivier Ricci curvature \cite{Oll-09,LLY-11} and Forman Ricci curvature \cite{For-03} based on Lipschitz contraction of heat semigroups also known as Dobrushin criterion \cite{DS-85}. For a general overview of various discrete curvature notions, we refer interested readers to \cite{NR-17}. More recently, a new notion of entropic curvature based on Schr{\"o}dinger bridges for the $W_1$-Wasserstein metric has been introduced in \cite{Sam-22,RS-23,Ped-23}. Other curvature notions were introduced in \cite{DL-22,Stein-23} based on the inverse of distance matrices.

We now shift our focus back to entropic curvature and Bakry-\'Emery curvature. Entropic curvature is a global curvature parameter associated to a given finite Markov chain, whereas Bakry-\'Emery curvature is defined locally on the vertices of a weighted graph. It is well-known that there is a connection between entropic curvature and Bakry-\'Emery curvature by replacing the logarithmic mean $\theta_{log}$ by the arithmetic mean $\theta_a$.
Therefore, it is natural to ask whether both curvature notions can be described, simultaneously, via a suitable $\Gamma$ calculus. 
The first part of this paper provides a positive answer to this question, leading to a general curvature-dimension inequality $CD_\theta(K,n)$ for finite Markov chains with a general mean $\theta$. 
This unifying viewpoint allows us to uncover some interesting further aspects related to both curvature notions: a natural way to introduce a dimension parameter into entropic curvature, an extension of Bakry-\'Emery curvature from vertices to measures, and the interpretation of the first positive Laplace eigenvalue $\lambda_1$ as a particular curvature notion.
Moreover, this approach would also allow to consider curvature-dimension conditions with other generators than the standard Laplacian, such as the magnetic or connection Laplacian (see \cite{LMP-19}). For gradient estimates of Bakry-\'Emery type on quantum Markov semigroups see the recent publication \cite{WZ-23}.

The second part of this paper begins with an equivalent formulation of the $CD_\theta(K,n)$ condition in terms of a gradient estimate, followed by some ``geometric'' applications of this gradient estimate. 
In particular, we derive entropic curvature results which are known to be true in the Bakry-\'Emery curvature setting: Bonnet-Myers type diameter estimates for entropic curvature as well as a proof of the fact that 
there are no expander graph families with non-negative entropic curvature. The latter follows from an inequality involving $\lambda_1$ and the average mixing time, and is inspired by Salez' original result \cite{Sal-22} for Ollivier Ricci curvature and Bakry-\'Emery curvature and the further developments presented in \cite{MS-23}. 

Let us now provide a more detailed description of the results in this paper.

\subsection{An adapted Gamma calculus to include entropic curvature}

All Markov chains $(X,Q,\pi)$ with transition probabilities $Q$ and stationary probability measure $\pi$ in this paper are irreducible, reversible and finite. We will always assume irreducibility and reversibility implicitly when we speak of \emph{finite Markov chains}. Our focus will be on two means, the \emph{logarithmic mean} 
$$ \theta_{log}(r,s) = \int_0^1 r^{1-p} s^p dp, $$
and the arithmetic mean
$$ \theta_a(r,s) = \frac{r+s}{2}. $$
For further details about these fundamental notions, in particular the definition of a general mean $\theta$, see Subsection \ref{subsec:markchain}.

Now we present a slightly modified $\Gamma$ calculus allowing to include both Bakry-\'Emery curvature and entropic curvature. For a given measure $\rho \in [0,\infty)^X$, the $\rho$-Laplacian for functions $f \in \R^X$ is defined as follows:
\begin{equation} \label{eq:Laprho}
\Delta_{\rho} f(x) = \sum_{y \in X} 2 \partial_1\theta(\rho(x),\rho(y)) (f(y)-f(x)) Q(x,y),
\end{equation}
where
$$ \partial_1 \theta(r,s) = \frac{\partial}{\partial r}\theta(r,s). $$
In the case of the logarithmic mean we require $\supp \rho = X$ in \eqref{eq:Laprho}, since in this case $\lim_{r \to 0} \partial_1\theta(r,s) = \infty$ for all $s > 0$. In the case of the arithmetic mean we have $\partial_1\theta \equiv \frac{1}{2}$, and $\Delta_{\rho}$ reduces for any choice of $\rho \in [0,\infty)^X$ to the standard Laplacian
\begin{equation} \label{eq:Lapstandard} 
\Delta f(x) = \sum_{y \in X} (f(y)-f(x)) Q(x,y). 
\end{equation}
The bilinear carr\'e du champ operators $\Gamma_\rho$ and $\Gamma_{2,\rho}$ on the Hilbert space $L^2(X,\pi)$ with inner product
\begin{equation} \label{eq:innprodpifunc} 
\langle f,g \rangle_\pi = \sum_{x \in X} f(x) g(x) \pi(x) 
\end{equation}
are given by
\begin{align*}
    2 \Gamma_{\rho}(f,g)(x) &:= \Delta_\rho(fg)(x) - f(x)\Delta_\rho g(x) - g(x)\Delta_\rho f(x), \nonumber \\
    2 \Gamma_{2,\rho}(f,g)(x) &:= \Delta \Gamma_\rho(f,g)(x) - \Gamma_\rho(f,\Delta g)(x) - \Gamma_\rho(g,\Delta f)(x)
\end{align*}
with their corresponding quadratic forms $\Gamma_\rho f = \Gamma_\rho(f,f)$ and $\Gamma_{2,\rho} f= \Gamma_{2,\rho}(f,f)$. Note
that $\Gamma_\rho$ involves the $\rho$-Laplacian, while $\Gamma_{2,\rho}$ involves $\Gamma_\rho$ and the standard Laplacian. Moreover, $\Gamma_\rho$ and $\Gamma_{2,\rho}$ coincide with the standard $\Gamma, \Gamma_2$ operators for any $\rho$, if we choose the arithmetic mean $\theta=\theta_a$. We use the following global curvature condition.

\begin{definition}[$CD_\theta(K,n)$]
  Let $(X,Q,\pi)$ be a finite Markov chain with a mean $\theta$. If $\theta$ satisfies $\theta(0,s) = 0$ for all $s > 0$, we set $I = I_\theta = (0,\infty)$ and if it satisfies $\theta(0,s) > 0$ for all $s > 0$, we set $I = I_\theta = [0,\infty)$. Let $n \in (0,\infty]$ and $K \in \R$. We say that $X$ satisfies the curvature-dimension inequality 
  $CD_\theta(K,n)$, if we have, for all $\rho \in I^X$,
  \begin{equation} 
  \label{eq:CDKntheta}
  \langle \rho,\Gamma_{2,\rho} f -\frac{1}{n} (\Delta f)^2 - K \Gamma_\rho f \rangle_\pi \ge 0 \quad \text{for all $f \in \R^X.$} 
  \end{equation}
  In the special cases $\theta_{log}$ and $\theta_a$ we also use the simplified notions $CD_{ent}(K,n)$ and $CD_a(K,n)$ for $CD_{\theta_{log}}(K,n)$ and $CD_{\theta_a}(K,n)$, respectively.
\end{definition}

Moreover, we also consider curvature notions for a fixed dimension parameter $n \in (0,\infty]$.

\begin{definition}[$K_n(X)$ and $K_n(\rho)$]\label{def:Kn}
  Let $(X,Q,\pi)$ be a finite Markov chain with a mean $\theta$ and $n \in (0,\infty]$ be fixed. We define
  $$ K_n(X) = K_{n,\theta}(X) := \sup \left\{ K \in \R: \text{\eqref{eq:CDKntheta} holds for $(K,n)$ and all $\rho \in (I_\theta)^X$}  \right\}. $$
  We call $K_n(X)$ the \emph{$n$-dimensional curvature of the Markov chain $X$}. 
  Moreover, if $\rho \in (I_\theta)^X$ is also fixed, we define
  $$ K_n(\rho) = K_{n,\theta}(\rho) := \sup \left\{ K \in \R: \text{\eqref{eq:CDKntheta} holds for $(K,n)$} \right\}, $$
  and call $K_n(\rho)$ the \emph{$n$-dimensional curvature of the measure $\rho$}.
\end{definition}

It follows immediately from this definition that
\begin{equation} \label{eq:KnXineq} 
K_n(X) = \inf_{\rho \in I^X} K_n(\rho). 
\end{equation}
In this general framework, we have the following facts for Markov chains $(X,Q,\pi)$, which are discussed and proved in Section \ref{sec:adaptGamma} of this paper:
\begin{itemize}
\item[(a)] $X$ has non-local entropic curvature $\ge K$ in the sense of Erbar-Maas (see \cite[Definition 4.2]{EM-12}) if and only if $X$ satisfies $CD_\theta(K,\infty)$ for the logarithmic mean $\theta=\theta_{log}$.
\item[(b)] The Bakry-\'Emery curvature $K_{x}(n)$ at a vertex $x \in X$, as defined e.g. in \cite[Definition 1.2]{CLP-19}, agrees with $K_{n,\theta}(\delta_x)$ in our Definition \ref{def:Kn}, where $\theta$ is the arithmetic mean $\theta_a$ and $\delta_x \in [0,\infty)^X$ is the Dirac measure at $x$, given by
\begin{equation} \label{eq:Diracx} 
\delta_x(y) = \begin{cases} 1/\pi(x) & \text{if $y = x$,} \\ 0 & \text{otherwise.} \end{cases}
\end{equation}
\item[(c)] The first positive eigenvalue $\lambda_1(X)$ of $- \Delta$ agrees with $K_{\infty,\theta}(\textbf{1}_X)$ for any mean $\theta$, where $\textbf{1}_X$ is the equilibrium measure $\textbf{1}_X(x) = 1$ for all $x \in X$.  
\end{itemize}
Statement (c) allows us to interpret $\lambda_1$ as a curvature value, and a straightforward consequence of (c) and \eqref{eq:KnXineq} is Lichnerowicz' Theorem:

\begin{theorem}[Lichnerowicz] 
  Let $(X,Q,\pi)$ be a finite Markov chain. Then we have for any mean $\theta$,
  \begin{equation} \label{eq:Lichineqintro} \lambda_1(X) \ge K_\infty(X),
  \end{equation}
  where $\lambda_1(X)$ is the first positive eigenvalue of $-\Delta$.
\end{theorem}

Markov chains, for which \eqref{eq:Lichineqintro} holds with equality, are called \emph{Lichnerowicz sharp}. Examples of Lichnerowicz sharp Markov chains are simple random walks on $N$-dimensional hypercubes in the case of the arithmetic mean (see, e.g., \cite{Schm-98} or \cite{KKRT-16} or \cite[Theorem 1.4]{LMP-22}) and in the case of the logarithmic mean (\cite[Corollary 1.4 and Section 7]{EM-12}, since they satisfy $\lambda_1(Q^N) = \frac{2}{N} = K_\infty(Q^N)$ in these cases.

In Section \ref{sec:optmeas}, we restrict our considerations entirely to Bakry-\'Emery curvature and the arithmetic mean $\theta_a$. In this case, we have
$$ K_n(X) = \min_{x \in X} K_n(\delta_x), $$
and we introduce the following optimality notion for measures.

\begin{definition}[Optimal measure] Let $(X,Q,\pi)$ be a finite Markov chain with the arithmetic mean $\theta=\theta_a$. A measure $\rho \in [0,\infty)^X$ is called \emph{optimal for dimension $n \in (0,\infty]$}, if
there exists $f_0 \in \R^X$ with $\Gamma f_0 > 0$ on $\supp \rho$ such that
$$ \langle \rho,\Gamma_2 f_0 - \frac{1}{n} (\Delta f_0)^2 - K_n(X) \Gamma f_0 \rangle_\pi = 0. $$
\end{definition}

Note that optimal measures $\rho$  satisfy
$$ K_n(\rho) = K_n(X), $$
but not vice versa (see Example \ref{ex:nonopt}). Examples of optimal measures are Dirac measures $\delta_x$ of vertices $x \in X$ with minimal $n$-dimensional Bakry-\'Emery curvature. Moreover, Lichnerowicz sharpness is equivalent to the property that the equilibrium $\textbf{1}_X$ is an optimal measure for dimension $n=\infty$ (see Proposition \ref{prop:Lichopt}). In Section \ref{sec:optmeas}, we will see that optimality of a measure $\rho$ is already determined by its support $\supp \rho\ \subset X$, which allows us to define optimality for subsets of $X$, leading to a simplicial structure on the vertices of $X$.

\subsection{Gradient estimate applications}

Recall the definition $I_\theta = (0,\infty)$ for means satisfying $\theta(0,s) = 0$ for all $s > 0$ and $I_\theta = [0,\infty)$ for means satisfying $\theta(0,s) > 0$ for all $s > 0$.
In Section \ref{sec:gradest}, we prove the following gradient estimate characterization of the curvature condition $CD_\theta(K,n)$.

\begin{theorem}[Global gradient estimate] Let $(X,Q,\pi)$ be a finite Markov chain with a mean $\theta$. The following statements are equivalent.
\begin{itemize}
\item[(i)] $X$ satisfies $CD_\theta(K,n)$.
\item[(ii)] For all measures $\rho \in (I_\theta)^X$ and $f \in \R^X$, we have
$$ \langle \rho, e^{-2Kt} P_t \Gamma_{P_t \rho} f - \Gamma_\rho P_t f \rangle_\pi \ge \frac{1-e^{-2Kt}}{Kn} \langle \rho, (\Delta P_t f)^2 \rangle_\pi, $$
\end{itemize}    
where $P_t: L^2(X,\pi) \to L^2(X,\pi)$ is the heat semigroup operator. In the case $K=0$ the constant on the right hand side is replaced by $\frac{2t}{n}$.
\end{theorem}

We present also in Section \ref{sec:gradest} a reverse Poincar{\'e} inequality for means $\theta \le \theta_a$ (see Theorem \ref{thm:revPoinc}; cf. \cite[Theorem 3.5]{EF-18}) as a consequence of this result.

The remainder of this paper is devoted to ``geometric'' applications of this gradient estimate. In Section \ref{sec:diambound}, we prove  Bonnet-Myers type diameter bounds with respect to the distance function $d_\Gamma: X \times X \to [0,\infty)$, given by
$$ d_\Gamma(x,y) := \sup \{ f(y)-f(x): \Vert \Gamma f \Vert_\infty \le 1\}. $$
Note, however, that these results provide also diameter bounds for the combinatorial distance function $d: X \times X \to [0,\infty)$, since both distance functions are related by (see \cite[Lemma 1.4]{LMP-18}):
$$ d(x,y) \le \sqrt{\frac{D}{2}} d_\Gamma(x,y) \le \frac{1}{\sqrt{2}} d_\Gamma(x,y), $$
where the second inequality follows from the definition of the maximal weighted vertex degree $D = \max_{x \in X} \sum_{y \neq x} Q(x,y) \le 1$. 

\begin{theorem}[Diameter bound under $CD_{ent}(K,\infty)$] If a finite Markov chain $(X,Q,\pi)$ satisfies $CD_{ent}(K,\infty)$ for some $K > 0$, then
$$ \diam(X,d_\Gamma) \le \frac{2}{K} \sqrt{\frac{2 D_\pi \log D_\pi}{D_\pi-1}}. $$
Here, $D_\pi$ denotes the maximal $\pi$-vertex degree, defined as follows:
$$ D_\pi := \max_{x \in X} \frac{1}{\pi(x)} \sum_{y \sim x} \pi(y), $$
where $y \sim x$ means $Q(x,y)>0$.
\end{theorem}

\begin{theorem}[Diameter bound under $CD_\theta(K,n)$]  Let $(X,Q,\pi)$ be a finite Markov chain with mean $\theta \le \theta_a$, satisfying $CD_\theta(K,n)$ for some $K > 0$ and $n \in (0,\infty)$.
Then we have
$$ \diam(X,d_\Gamma) \le \pi \sqrt{\frac{n}{K}}.
$$
\end{theorem}

These results complement the following Bonnet-Myers type diameter bound for entropic curvature:

\begin{theorem}[see {\cite[Proposition 7.3]{EF-18}}] \label{thm:BMEF-18}
  Let $(X,Q,\pi)$ be a finite Markov chain satisfying $CD_{ent}(K,\infty)$ for some $K > 0$. Then we have
  $$ \diam(X,d_{\mathcal{W}}) \le 2 \sqrt{\frac{-2 \log \pi_{\min}}{K}} $$
  with $\pi_{\min} = \min_{x \in X} \pi(x)$ and 
  $d_{\mathcal{W}}(x,y) := \mathcal{W}(\delta_x,\delta_y)$.
\end{theorem}

Theorem \ref{thm:BMEF-18} provides also a diameter bound for the combinatorial distance function, due to the following relation (see \cite[Proposition 2.12 and Lemma 2.13]{EM-12}):
$$ \sqrt{2} d(x,y) \le d_{\mathcal{W}}(x,y) \le \frac{1.57}{Q_{\min}}d(x,y) $$
with $Q_{\min} := \min_{x \sim y} Q(x,y)$. Note that Theorem \ref{thm:BMEF-18} involves, besides the curvature parameter $K$, the value $\pi_{\min}$, which is related to the size $|X|$ of the Markov chain, which is a global value. In the case of the $N$-dimensional hypercube $Q^N$, the upper bound of the diameter grows like $O(N)$, which is the right order.
Our Bonnet-Myers type results are not of the correct order for hypercubes, but they do not involve the overall size of the Markov chain and only depend on local properties like the vertex degrees.  

The final result concerns expander graphs. A problem mentioned by Ollivier \cite[Problem T]{Oll-10} and attributed to Naor and Milman asks whether there exist expander graph families of non-negative Ollivier Ricci curvature. An analogous question for Bakry-\'Emery curvature was raised in \cite[Conjecture 9.11]{CLP-19}. This question was settled by Salez in \cite{Sal-22} for both curvature notions, where he proved that no large enough non-negatively curved combinatorial graph $G$ can be simultaneously sparse and have 
a uniform lower bound on its spectral gap. A quantitative version of this result, which is almost sharp, was later given in \cite{MS-23}. Since the essential ingredient of this quantitative result is a gradient estimate of the above type, the arguments in \cite{MS-23} can be adapted to the case of non-negative entropic curvature. Key inequalities are a lower bound on the average mixing time, given in Lemma \ref{lem:tau} of Subsection \ref{subsec:cheegavmix}, and an upper bound on the product of $\lambda_1(X)$ and the average mixing time, given in Theorem \ref{thm:lambda-tau} in Subsection \ref{subsec:specgapnonegcurv}. These results lead the following explicit spectral gap estimate.

\begin{theorem}[Spectral gap estimate]
    Let $G = (X,E)$ be a finite, connected, $d$-regular combinatorial graph with vertex set $X$ satisfying $|X| \ge 4d$ and edge set $E$. Let $Q(x,y) = \frac{1}{d}$ for $y \sim x$ and $\pi(x) = 1/|X|$ (simple random walk). Assume that
    the associated Markov chain $(X,Q,\pi)$ satisfies $CD_{ent}(0,\infty)$. Then we have
    $$ d-\mu_2(G) \le \frac{4000\, d^4 \log d}{\log(|X|/4)}, $$
    where $\mu_2(G)$ is the second largest eigenvalue of the adjacency matrix.
    
    In particular, we have $d-\mu_2(G_k) \to 0$ for families of connected $d$-regular graphs $G_k=(V_k,E_k)$ satisfying $CD_{ent}(0,\infty)$ with $|V_k| \to \infty$, that is, there are no expanders of non-negative entropic curvature.
\end{theorem}

\section{Overview over entropic curvature and Bakry-\'Emery curvature}
\label{sec:overview}

\subsection{Markov chains and means}
\label{subsec:markchain}

Henceforth $(X,Q,\pi)$ will always represent an irreducible, reversible finite Markov chain. In this article, we will refer to them simply as \emph{finite Markov chains} by tacitly assuming these extra properties of irreducibility and reversibility. The irreducible Markov kernel $Q(x,y) \in [0,1]$ describes the transition probability from $x\in X$ to $y \in X$ and satisfies $\sum_{y \in X} Q(x,y) = 1$ for all $x \in X$. The reference measure $\pi \in (0,1]^X$ is the unique stationary probability measure satisfying
$$ \sum_{x \in X} \pi(x) = 1 \quad \text{and} \quad \pi(y) = \sum_{x \in X} \pi(x) Q(x,y), $$ 
and reversibility implies that $Q(x,y)\pi(x)=Q(y,x)\pi(y)$. Let $w: X \times X \to [0,\infty)$ be the symmetric edge weight, given by $w(x,y) = Q(x,y)\pi(x)$ for $x \neq y$ and $w(x,x) = 0$. We write $x \sim y$ if $w(x,y) > 0$ and call $x$ and $y$ neighbours in $X$. In this way, $X$ inherits the structure of a combinatorial graph and it makes sense to speak of balls $B_r(x)$ and spheres $S_r(x)$ of radius $r \in \mathbb{N} \cup \{0\}$ around $x \in X$. Irreducibility automatically implies that the underlying graph $X$ is connected. We define the following weighted vertex degree for a vertex $x \in X$: 
\begin{equation} \label{eq:weighdeg} 
D(x) = \frac{\sum_{y \in X} w_{xy}}{\pi(x)} = \sum_{y \neq x} Q(x,y) \le 1, 
\end{equation}

\medskip

The general definition of a mean was given in \cite{Ma-11,EM-12}. We recall it with the omission of one property, which we will discuss directly afterwards. 

\begin{definition}[Mean, see \cite{Ma-11,EM-12}] \label{def:mean}
A function $\theta: [0,\infty) \times [0,\infty) \to [0,\infty)$ is called a \emph{mean} if the following properties are satisfied:
\begin{itemize}
\item[(i)] $\theta$ is continuous and its restriction to $(0,\infty) \times (0,\infty)$ is smooth; 
\item[(ii)] $\theta$ is symmetric:  $\theta(r,s) = \theta(s,r)$; 
\item[(iii)] $\theta$ is monotone: $\theta(r,s) \ge \theta(t,s)$ for $r \ge t$;
\item[(iv)] $\theta$ is homogeneous: $\theta(\lambda r,\lambda s) = \lambda \theta(r,s)$;
\item[(v)] $\theta$ is normalized: $\theta(1,1) = 1$.
\end{itemize}
The \emph{logarithmic mean} $\theta_{log}$ is defined as follows:
$$ \theta_{log}(r,s) = \int_0^1 r^{1-p}s^p dp $$
and has the additional property
\begin{itemize}
\item[(vi)] $\theta(0,s) = 0$ for all $s \ge 0$.     
\end{itemize}
\end{definition}

The logarithmic mean $\theta_{log}$ is an important concept for entropic curvature and its name stems from the fact that we have for all $r,s > 0$ with $r \neq s$,
\begin{equation} \label{eq:logmeanformula} 
\theta_{log}(r,s) = \frac{r-s}{\log(r)-\log(s)}. 
\end{equation}
In contrast to our definition, the definition in \cite{Ma-11,EM-12} requires all six conditions (i)-(vi). 
Since we want the arithmetic mean $\theta_a$, given by
$$ \theta_a(r,s) = \frac{r+s}{2}, $$
to be included, we dropped the last condition (vi), which is not satisfied by $\theta_a$. Another natural mean is the geometric mean $\theta_g$, given by
$$ \theta_g(r,s) = \sqrt{rs}. $$
The geometric mean satisfies condition (vi) and these three means are ordered by the following inequalities:
$$ \min\{r,s\} \le \theta_g(r,s) \le \theta_{log}(r,s) \le \theta_a(r,s) \le \max\{r,s\}. $$
Note that $\min$ and $\max$ are not means in the sense of Definition \ref{def:mean} since the smoothness condition of property (i) is violated.

For the partial derivatives of a mean $\theta$, we use the simplified notation
\begin{equation} \label{eq:partial12theta}
\partial_1 \theta(r,s) = \frac{\partial}{\partial r} \theta(r,s) \quad \text{and} \quad \partial_2 \theta(r,s) = \frac{\partial}{\partial s} \theta(r,s). 
\end{equation}
We note that we have for any mean $\theta$,
$$ \partial_1 \theta(r,s) = \partial_2 \theta(s,r) $$
and
$$ \theta(r,s) = \frac{d}{d\lambda}\Big\vert_{\lambda=1} \theta(\lambda r,\lambda s) = r \partial_1\theta(r,s) + s \partial_2\theta(r,s), $$
and therefore,
$$ \partial_1 \theta(r,r) = \partial_2 \theta(r,r) = \frac{1}{2} $$
and (see \cite[Lemma 2.2]{EM-12}), for any measure $\rho \in (I_\theta)^X$ and any pair $x,y \in X$,
\begin{equation} \label{eq:hatderid} 
\hat \rho_{xy} := \theta(\rho(x),\rho(y)) = \partial_1 \theta(\rho(x),\rho(y))\rho(x) + \partial_2 \theta(\rho(x),\rho(y))\rho(y). 
\end{equation}
Any finite Markov chain together with a mean gives rise to an associated curvature notion, as explained in Subsection \ref{subsec:curvnot}.

\subsection{Entropic curvature}

For the original definition and more background on entropic curvature, see \cite{EM-12}. See also \cite{Ka-21} for an overview over this notion and various other curvature notions.
Entropic curvature of $(X,Q,\pi)$ can be viewed as a global lower Ricci curvature bound of the space and is defined through the geometry of the space $$ \mathcal{P}(X) = \{ \rho \in [0,\infty)^X:\,  \sum_{x \in X} \rho(x) \pi(x) = 1 \} $$
of probability densities of $X$ (with respect to the reference measure $\pi$). 

To this end, we first introduce some vector analysis on the space $X$ and fundamental inner products on functions and vector fields. A \emph{vector field} on $X$ is a function $V: X \times X \to \mathbb{R}$ satisfying $V(x,y) = - V(y,x)$ and $V(x,y) \neq 0$ only if $x \sim y$. The space of all vector fields on $X$ is denoted by 
$\mathcal{X}(X)$. The \emph{gradient} of a function $f \in \R^X$ is the vector field $\nabla f$, given by
$$ \nabla f(x,y) := \begin{cases} f(y)-f(x) & \text{if $x \sim y$,} \\ 0 &\text{if $x \not\sim y$.} \end{cases} $$
The \emph{divergence} of a vector field $V \in \mathcal{X}(X)$ is defined by
$$ {\rm{div}}\, V(x) = \sum_{y \in X} V(x,y) Q(x,y). $$
Finally, the standard \emph{Laplacian} for functions is given by 
(cf. \eqref{eq:Lapstandard})
$$
\Delta f(x) = {\rm{div}}\, (\nabla f)(x) = \sum_{y \in X} Q(x,y) (f(y)-f(x)).
$$
For $f_1, f_2 \in \mathbb{R}^X$ and $V_1,V_2 \in \mathcal{X}(X)$, we define the following inner products (cf. \eqref{eq:innprodpifunc}):
\begin{align*}
\langle f_1,f_2 \rangle_\pi &:= \sum_{x \in X} f_1(x) f_2(x) \pi(x), \\
\langle V_1,V_2 \rangle_\pi &:= \frac{1}{2} \sum_{x,y \in X} V_1(x,y)V_2(x,y)Q(x,y)\pi(x).
\end{align*}
The operators ${\rm{div}}$ and $-\nabla$ are adjoint operators with respect to these inner products, i.e.,
$$ \langle f_1, \dvg\, V_1 \rangle_\pi = -\langle \nabla f_1,V_1 \rangle_\pi. $$

For any function $\rho \in [0,\infty)^X$ we define $\hat \rho: X \times X \to \mathbb{R}$ as
$$ \hat \rho(x,y) = \hat \rho_{xy} = \theta(\rho(x),\rho(y)), $$
and we introduce the following inner product on $\mathcal{X}(X)$ associated to $\rho$:
$$ \langle V_1,V_2 \rangle_\rho := \langle \hat \rho\, V_1,V_2 \rangle_\pi = \frac{1}{2} \sum_{x,y \in X} \hat \rho_{xy} V_1(x,y) V_2(x,y) Q(x,y) \pi(x). $$

A metric $\mathcal{W}: \mathcal{P}(X) \times \mathcal{P}(X) \to [0,\infty)$ 
was introduced in \cite{Ma-11,EM-12} as follows: For $\rho_0,\rho_1 \in \mathcal{P}(X)$, let
$$ \mathcal{W}(\rho_0,\rho_1) := \inf\left\{ \int_0^1 \Vert \nabla \psi_t \Vert_{\rho_t}^2 dt\, \vert\, (\rho_t,\psi_t) \in CE(\rho_0,\rho_1) \right\}^{\frac{1}{2}}, $$
where $CE(\rho_0,\rho_1)$ consists of all smooth curves $\rho: [0,1] \to \mathcal{P}(X)$ with boundary conditions $\rho(0)=\rho_0$ and $\rho(1)=\rho_1$, together with measurable $\psi: [0,1] \to \mathbb{R}^X$ which satisfy the \emph{discrete continuity equation}: 
$$ 
\frac{d}{dt}\rho_t+{\rm{div}}(\hat \rho_t \nabla \psi_t) = 0. 
$$
Equipped with this metric, the space $\mathcal{P}(X)$ is complete (see \cite[Theorem 2.4]{EM-12}) and induced by a Riemannian metric on the interior
$$ \mathcal{P}_*(X) = \{ \rho \in \mathcal{P}(X) \mid \text{$\rho(x) > 0$ for all $x \in X$} \} $$
of $\mathcal{P}(X)$, given by the Riemannian metric $g_\rho(\nabla \psi_1,\nabla \psi_2) := \langle \nabla \psi_1,\nabla \psi_2 \rangle_\rho$ at $\rho \in \mathcal{P}_*(X)$.
For a given curve $\rho_t$, a solution $\nabla \psi_t$ of the discrete continuity equation can be viewed as its tangent vector field, and the distance $\mathcal{W}(\rho_0,\rho_1)$ is then derived via the energy of smooth curves connecting $\rho_0$ and $\rho_1$.
The definition of $\mathcal{W}$ is then a discrete analogue of the Benamou-Brenier formula \cite{BB-00}.
Moreover, any pair of probability densities $\rho_0, \rho_1 \in \mathcal{P}(X)$ can be joined by a constant speed geodesic $\rho_t$ (see \cite[Theorem 3.2]{EM-12}). That is, we have $\rho(0)=\rho_0$, $\rho(1)=\rho_1$ and
$$ \mathcal{W}(\rho_s,\rho_t) = |s-t| \mathcal{W}(\rho_0,\rho_1) \quad \text{for all $0 \le s,t \le 1$.} $$
The final concept relevant for the definition of entropic curvature is the \emph{entropy functional} (relative to the reference measure $\pi$), which is defined for any probability density $\rho \in \mathcal{P}(X)$ by
$$ {\rm{Ent}}(\rho) := \langle \rho, \log \rho \rangle_\pi = \sum_{x \in X} \rho(x) \log \rho(x) \pi(x). $$

\begin{definition}[Entropic curvature (see {\cite[Definition 1.1]{EM-12}})
] 
A finite Markov chain $(X,Q,\pi)$ has \emph{entropic
curvature bounded below by $K \in \mathbb{R}$}, if we have for every constant speed geodesic $(\rho_t)_{t \in [0,1]}$ in $(\mathcal{P}(X),\mathcal{W})$,
\begin{equation} \label{eq:ent-curv} 
{\rm{Ent}}(\rho_t) \le (1-t){\rm{Ent}}(\rho_0) + t {\rm{Ent}}(\rho_1) - K \frac{t(1-t)}{2} \mathcal{W}(\rho_0,\rho_1)^2 
\end{equation}
for all $t \in [0,1]$.
\end{definition}

Various equivalent conditions to the one in this definition are given in \cite{EM-12}. We will
present some of them in the following theorem. 
Recall the following earlier introduced notation for the partial derivatives:
$$ \partial_1 \theta(r,s) = \frac{\partial}{\partial r} \theta(r,s) \quad \text{and} \quad \partial_2 \theta(r,s) = \frac{\partial}{\partial s} \theta(r,s). $$
Moreover, we define for $\rho \in \mathcal{P}_*(X)$ and $f \in \mathbb{R}^X$ (see \cite[page 9 and (4.2)]{EM-12}):
\begin{align*}
  (\widehat \Delta \rho)(x,y) &:= \partial_1 \theta(\rho(x),\rho(y)) (\Delta \rho)(x) + \partial_2 \theta(\rho(x),\rho(y)) (\Delta \rho)(y), \\
  \mathcal{A}_\rho(f) &:= \Vert \nabla f \Vert_\rho^2, \\
  \mathcal{B}_\rho(f) &:= \frac{1}{2} \langle \widehat \Delta \rho\cdot\nabla f,\nabla f \rangle_\pi - \langle \hat \rho \cdot \nabla f, \nabla (\Delta f) \rangle_\pi.
\end{align*}

\begin{theorem}[{\cite[Theorem 4.5]{EM-12}}] \label{thm:EM4.5} Let $(X,Q,\pi)$ be a finite Markov chain and $K \in \mathbb{R}$. The following are equivalent:
\begin{itemize}
    \item[(a)] 
    $(X,Q,\pi)$ has entropic curvature bounded below by $K$. 
    \item[(b)] For every constant speed geodesic $(\rho_t)_{t \in [0,1]}$ with end-points in $\mathcal{P}_*(X)$ we have
    $$ \frac{d^2}{dt^2} {\rm{Ent}}(\rho_t) \ge K\, \mathcal{W}(\rho_0,\rho_1)^2. $$
    \item[(c)] For all $\rho \in \mathcal{P}_*(X)$ and $f \in \mathbb{R}^X$,
    $$ \mathcal{B}_\rho(f) \ge K \mathcal{A}_\rho(f).$$
\end{itemize}
\end{theorem}

While (a) is the condition of geodesic $K$-convexity, condition (b) is the corresponding infinitesimal formulation of $K$-convexity, and (c) can be viewed as a Bochner-type formula.
It was shown in Mielke \cite[Theorem 4.1]{Mi-13} 
that every finite Markov chain has entropic curvature bounded below by some $K > -\infty$. This is by no means a simple fact.
%generators.

\subsection{Bakry-\'Emery curvature}

Bakry-\'Emery curvature is an analytic curvature notion motivated by Bochner's identity
$$ \frac{1}{2} \Delta \Vert \nabla f \Vert^2 = \Vert\Hess f \Vert_{HS}^2 + g(\nabla f, \nabla \Delta f) + \Ric(\nabla f) $$
for any Riemannian manifold 
$(M,g)$. (Here $\Vert \cdot \Vert_{HS}$ denotes the Hilbert-Schmidt norm). We say that Ricci curvature of $M$ is \emph{bounded from below by $K\in \mathbb{R}$} if $\Ric(v)\ge K\Vert v \Vert^2$ for all tangent vectors $v$. Moreover, the Hessian term satisfies $\Vert \Hess f \Vert_{HS}^2 \ge \frac{1}{n} (\Delta f)^2$, where $n$ is the dimension of $M$. The Bochner's identity then leads to the inequality
$$ \frac{1}{2} \Delta \Vert \nabla f \Vert^2 - g(\nabla f,\nabla \Delta f) \ge \frac{1}{n} (\Delta f)^2 + K\Vert \nabla f\Vert^2, $$
which can be formulated, alternatively, via Bakry-\'Emery's $\Gamma$-operators (see \cite{BE-85})
\begin{align}
\Gamma(f_1,f_2) &:= \frac{1}{2} (\Delta(f_1f_2) - f_1 \Delta f_2 - f_2 \Delta f_1) = g(\nabla f_1,\nabla f_2), \label{eq:gamma}\\
\Gamma_2(f_1,f_2) &:= \frac{1}{2} (\Delta \Gamma(f_1,f_2) - \Gamma(f_1,\Delta f_2) - \Gamma(f_2,\Delta f_1)) \label{eq:gamma2}
\end{align}
as the following inequality: 
$$ \Gamma_2 (f,f) \ge \frac{1}{n} (\Delta f)^2 + K \Gamma (f,f). $$
%We remark that a complete $n$-dimensional Riemannian manifold satisfies $CD(K,n)$ if and only if the Ricci curvature of the manifold is bounded below by $K$.

Bakry-\'Emery curvature is defined in the setting of finite Markov chains as follows.

\begin{definition}[Bakry-\'Emery curvature] \label{def:BEcurv}
Let $(X,Q,\pi)$ be 
a finite Markov chain. Recall the definition of $\Delta$, $\Gamma$ and $\Gamma_2$:  for $f,g \in  \mathbb{R}^X$,
\begin{align*}
\Delta f(x) &= \sum_{y \in X} (f(y)-f(x)) Q(x,y),\\
2 \Gamma(f,g) &= \Delta(fg) - f\Delta g - g \Delta f, \\
2 \Gamma_2(f,g) &= \Delta\Gamma(f,g) - \Gamma(f,\Delta g) - \Gamma(g,\Delta f),
\end{align*}
with shortened notation $\Gamma f:=\Gamma(f,f)$ and $\Gamma_2 f:=\Gamma_2(f,f)$.

For a given dimension parameter $n \in (0,\infty]$, the \emph{Bakry-\'Emery curvature $K_n(x)$ at $x \in X$} is the supremum of all $K \in \mathbb{R}$ satisfying
$$ \Gamma_2 f(x) \ge \frac{1}{n} (\Delta f(x))^2  + K \Gamma f(x) \quad \text{for all $f \in \mathbb{R}^X$.} $$
We say that $(X,Q,\pi)$ satisfies the global \emph{curvature-dimension condition $CD(K,n)$} if we have
$$ K_n(x) \ge K \quad \text{for all $x \in X$}. $$
\end{definition}

An easy computation shows 
$$ 2 \Gamma(f,g)(x) = \sum_{y \sim x} Q(x,y) (f(y)-f(x))(g(y)-g(x)), $$
which implies $\Gamma f(x) \ge 0$. Moreover, if $(X,Q,\pi)$ satisfies $CD(K,n)$, then it also satisfies $CD(K',n')$ for any $K' \le K$ and $n' \ge n$. It was shown in \cite[Lemma 1]{Elw-91} that every finite combinatorial graph $G=(V,E)$ with the simple random walk given by $Q(x,y) = \frac{1}{d_x}$ satisfies $CD(-1,2)$.

Let us finally mention the following trade-off between curvature and dimension in the case of Bakry-\'Emery curvature.

\begin{remark}[Trade-off between curvature and dimension] For any vertex $x \in X$ of a finite Markov chain $(X,Q,\pi)$, the function $n \mapsto K_n(x)$ is monotone non-decreasing. It follows from (see \cite[Section 2.1]{LMP-18})
$$ (\Delta f(x))^2 \le 2 D(x) \Gamma (x) $$
that we have, for $0 < n' \le n \le \infty$, 
$$ 0 \le K_n(x) - K_{n'}(x) \le 2 D(x) \left( \frac{1}{n'} - \frac{1}{n} \right). $$
See also \cite[Lemma 4.3]{CLP-19} for this relation in the case of the non-normalized Laplacian
on a combinatorial graph.    
\end{remark}

\section{Entropic curvature via Gamma calculus}
\label{sec:adaptGamma}

Entropic curvature can be formulated equivalently via an integrated Bochner formula and is therefore non-local in nature. On the other hand, Bakry and \'Emery introduced local $\Gamma$-operators in \cite{BE-85}, leading to a curvature-dimension inequality which is closely related to a pointwise Bochner formula. 

In this section, we slightly modify the Bakry-\'Emery $\Gamma$ calculus, which allows us to introduce the non-local entropic curvature with the help of these local $\Gamma$-operators. Moreover, this leads also naturally to the introduction of a dimension parameter for entropic curvature.

\subsection{A modified Gamma calculus}

Recall the definition $I_\theta = (0,\infty)$
 and $I_\theta = [0,\infty)$ for a given mean $\theta$, depending on whether $\theta(0,s) = 0$ for all $s > 0$ or $\theta(0,s) > 0$ for all $s > 0$.  

\begin{definition}[$\Gamma_\rho$ and $\Gamma_{2,\rho}$] \label{def:gammarho}
  Let $(X,Q,\pi)$ be a finite Markov chain and $\theta$ a mean (see Definition \ref{def:mean}). For a given measure $\rho \in (I_\theta)^X$, the Laplacian $\Delta_\rho$ and the operators $\Gamma_\rho$ and $\Gamma_{2,\rho}$ are defined as follows for functions $f,f_1,f_2 \in R^X$:
  \begin{align*}
     \Delta_{\rho} f(x) &:=  \sum_{y \in X} 2 \partial_1\theta(\rho_x,\rho_y)\, (f(y)-f(x))\, Q(x,y), \\
    2 \Gamma_\rho(f_1,f_2)(x) &:=
    \Delta_\rho (f_1f_2)(x) - (f_1 \Delta_\rho f_2)(x) - (f_2 \Delta_\rho f_1)(x), \\
    2 \Gamma_{2,\rho}(f_1,f_2)(x) &:= \Delta \Gamma_{\rho}(f_1,f_2)(x) - \Gamma_\rho(f_1,\Delta f_2)(x) - \Gamma_\rho(f_2,\Delta f_1)(x),
  \end{align*}
  where we used the notation in \eqref{eq:partial12theta} for the partial derivative of $\theta$ and $\rho_z := \rho(z)$. 

  Moreover, we set $\Gamma_\rho f := \Gamma_\rho(f,f)$ and $\Gamma_{2,\rho} f := \Gamma_{2,\rho}(f,f)$.
\end{definition}

As in the case of the classical Bakry-\'Emery operator $\Gamma$, we have for any measure $\rho \in (I_\theta)^X$,
\begin{equation} \label{eq:Gammarhogreen}
\Gamma_\rho (f_1,f_2)(x) = \sum_{y \in X} \partial_1\theta(\rho_x,\rho_y)(f_1(y)-f_1(x))(f_2(y)-f_2(x))Q(x,y). 
\end{equation}

\begin{remark} 
The geodesic equations in \cite[Theorem 3.31]{Ma-11} can also be written in terms of $\Gamma_\rho$ as follows: 
\begin{align*}
    \partial_t\rho_t &= - {\rm{div}}(\hat \rho_t \nabla \psi_t), \\
    \partial_t \psi_t &= \frac{1}{2}\Gamma_{\rho_t}(\psi_t,\psi_t).
\end{align*}
\end{remark}

\begin{remark}
Note that $\Delta_\rho$ agrees with the standard Laplacian $\Delta$ in the following two cases:
\begin{itemize}
    \item If $\rho$ is the equilibrium $\textbf{1}_X$ and $\theta$ is an arbitrary mean, since $\partial_1 \theta(r,r) = \frac{1}{2}$,
    \item If $\theta$ is the arithmetic mean $\theta_a$ and $\rho$ is an arbitrary measure, since we have
    $$ \partial_1 \theta_a \equiv \frac{1}{2}. $$
\end{itemize}
Consequently, the operators $\Gamma_\rho$ and $\Gamma_{2,\rho}$ reduce to the classical Bakry-\'Emery operators $\Gamma$ and $\Gamma_2$, given in \eqref{eq:gamma} and \eqref{eq:gamma2}, respectively, in these two cases.
\end{remark}

Proposition \ref{prop:geommeansa} below provides an amusing fact about the geometric mean. Before we present it, we need to introduce the following inner products, associated to a measure $\rho \in (I_\theta)^X$:
\begin{align}
    \langle f_1,f_2 \rangle_{\rho\cdot \pi} &=
    \sum_{x \in X} f_1(x)f_2(x) \rho_x \pi(x), \label{eq:inprodrhopi}\\
    \langle V_1,V_2 \rangle_\rho &= \langle \hat \rho \cdot V_1,V_2 \rangle_\pi = \frac{1}{2} \sum_{x,y \in X} \hat \rho_{xy} V_1(x,y)V_2(x,y) Q(x,y)\pi(x), \nonumber
\end{align}
with $\hat \rho: X \times X \to \R$ given by
\begin{equation} \label{eq:hatderid2} 
\hat \rho_{xy} = \hat \rho(x,y) := \theta(\rho_x,\rho_y). 
\end{equation}

\begin{proposition} \label{prop:geommeansa}
  Let $(X,Q,\pi)$ be a finite Markov chain, 
  $$ \theta(r,s) = \theta_g(r,s) = \sqrt{rs} $$ 
  be the geometric mean and $\rho \in (0,\infty)^X$ be a measure. Then we have the following Green's formula for functions $f_1, f_2 \in \R^X$:
  $$ \langle \Delta_\rho f_1,f_2 \rangle_{\rho\cdot\pi} = - \langle \nabla f_1, \nabla f_2 \rangle_\rho. $$
  In particular, $\Delta_\rho$ is a self-adjoint operator with respect to the inner product \eqref{eq:inprodrhopi} in the case of the geometric mean.
\end{proposition}

\begin{proof}
    In the case of the geometric mean $\theta = \theta_g$ we have for $r,s > 0$,
    $$ \partial_1 \theta(r,s) \cdot r = \frac{1}{2} \theta(r,s) = \partial_2 \theta(r,s) \cdot s. $$
    Using this and $Q(x,y)\pi(x) = Q(y,x)\pi(y)$, we have
    \begin{align*}
    \langle \Delta_\rho f_1,f_2 \rangle_{\rho\cdot \pi} &= \sum_{x,y \in X}
    2 \partial_1\theta(\rho_x,\rho_y)\rho_x(f_1(y)-f_1(x)) f_2(x) Q(x,y) \pi(x) \\
    &=  \sum_{x,y \in X} \theta(\rho_x,\rho_y) (f_1(y)-f_1(x)) f_2(x) Q(x,y) \pi(x) \\
    &= - \frac{1}{2} \sum_{y,x \in X} \theta(\rho_x,\rho_y) (f_1(y)-f_1(x)) (f_2(y)-f_2(x) Q(x,y) \pi(x) \\  
    &= - \langle \nabla f_1, \nabla f_2 \rangle_\rho.
    \end{align*}
\end{proof}

Now we introduce, in accordance with \cite[page 9]{EM-12}, the operator $\mathcal{A}_\rho$ as follows:
\begin{equation} \label{eq:Arhodef} 
\mathcal{A}_\rho(f) := \Vert \nabla f \Vert_\rho^2 = \langle \hat \rho \cdot \nabla f, \nabla f \rangle_\pi. \end{equation}

\begin{lemma}[Properties of $\mathcal{A}_\rho$]
\label{lem:Arhoprop}
Let $(X,Q,\pi)$ be a finite Markov chain with mean $\theta$. Then we have for any measure $\rho \in (I_\theta)^X$,
\begin{equation} \label{eq:Agamma} 
\mathcal{A}_\rho(f) = \langle \rho, \Gamma_\rho f \rangle_\pi. 
\end{equation}
Moreover, if the mean $\theta$ satisfies $\theta \le \theta_a$,
then we have the estimate
\begin{equation} \label{eq:Agammaest} 
\mathcal{A}_\rho(f) \le \langle \rho,\Gamma f \rangle_\pi. 
\end{equation}
Inequality \eqref{eq:Agammaest} holds in particular for all means satisfying $\theta(0,s) =0$ for all $s > 0$.
\end{lemma}

\begin{proof}
    Using \eqref{eq:Gammarhogreen}, we obtain for $f_1,f_2 \in \R^X$,
\begin{align} \label{eq:nablaGammarho} 
\langle \nabla f_1, \nabla f_2 \rangle_\rho 
&= \langle \hat \rho \cdot \nabla f_1, \nabla f_2 \rangle_\pi \nonumber \\ 
&= \frac{1}{2} \sum_{x,y \in X} \hat \rho_{xy} (f_1(y)-f_1(x))(f_2(y)-f_2(x)) Q(x,y) \pi(x) \nonumber \\
&\stackrel{\mathclap{\eqref{eq:hatderid}}}{=} \frac{1}{2} \sum_{x \in X} \rho_x \sum_{y \in X} \partial_1\theta(\rho_x,\rho_y) (f_1(y)-f_1(x))(f_2(y)-f_2(x)) Q(x,y) \pi(x) \nonumber \\ 
&+ \frac{1}{2} \sum_{y \in X} \rho_y \sum_{x \in X} \underbrace{\partial_2 \theta(\rho_x,\rho_y)}_{=\partial_1\theta(\rho_y,\rho_x)} (f_1(x)-f_1(y))(f_2(x)-f_2(y)) \underbrace{Q(x,y) \pi(x)}_{= Q(y,x) \pi(y)} \nonumber \\
&= \sum_{x \in X} \rho_x \sum_{y \in X} \partial_1\theta(\rho_x,\rho_y) (f_1(y)-f_1(x))(f_2(y)-f_2(x)) Q(x,y) \pi(x) \nonumber \\ 
&\stackrel{\mathclap{\eqref{eq:Gammarhogreen}}}{=} \sum_{x \in X} \rho(x) \Gamma_\rho(f_1,f_2)(x) \pi(x) = \langle \rho, \Gamma_\rho(f_1,f_2)
\rangle_\pi.
\end{align}
This shows \eqref{eq:Agamma}.
The proof of \eqref{eq:Agammaest} follows from
\begin{align*}
\mathcal{A}_\rho (f) &= \frac{1}{2} \sum_{x,y \in X} \theta(\rho_x,\rho_y) (f(y)-f(x))^2 Q(x,y)\pi(x)  \\
&\leq  \frac{1}{2} \sum_{x,y \in X} \frac{\rho_x+\rho_y}{2}  (f(y)-f(x))^2 Q(x,y)\pi(x) \\
&= \frac{1}{2} \sum_{x,y \in X} \rho(x) (f(y)-f(x))^2 Q(x,y)\pi(x) =\langle \rho, \Gamma f \rangle_\pi.
\end{align*}
Finally, we note that the condition $\theta(0,s) = 0$ for all $s > 0$ implies (see \cite[(2.1)]{EM-12})
$$ \theta(r,s) \le \theta_a(r,s) = \frac{r+s}{2}. $$
\end{proof}

Now we introduce the operator $\mathcal{B}_\rho$
in accordance with \cite[(4.2)]{EM-12}. For a measure $\rho \in (I_\theta)^X$, let
$\hat \Delta \rho$ be defined as
\begin{equation} \label{eq:hatDelta} 
\hat \Delta \rho(x,y) := \partial_1\theta(\rho_x,\rho_y) \Delta \rho(x) + \partial_2\theta(\rho_x,\rho_y) \Delta \rho(y). 
\end{equation}
Then $\mathcal{B}_\rho$ is the following operator:
\begin{equation} \label{eq:Brhodef}
\mathcal{B}_\rho(f) := \frac{1}{2} \langle \hat \Delta \rho \cdot \nabla f, \nabla f \rangle_\pi - \langle \hat \rho \cdot \nabla f, \nabla (\Delta f) \rangle_\pi.
\end{equation}

\begin{lemma}[Properties of $\mathcal{B}_\rho$]
\label{lem:Brhoprop}
Let $(X,Q,\pi)$ be a finite Markov chain with mean $\theta$. Then we have for any measure $\rho \in (I_\theta)^X$,
\begin{equation} \label{eq:cruc} 
\langle \hat \Delta \rho \cdot \nabla f, \nabla f \rangle_\pi = \langle \rho, \Delta (\Gamma_\rho f) \rangle_\pi, 
\end{equation}
and therefore,
$$ \mathcal{B}_\rho(f) = \langle \rho, \frac{1}{2}\Delta (\Gamma_\rho f) - \Gamma_\rho(f,\Delta f) \rangle_\pi = \langle \rho, \Gamma_{2,\rho} f \rangle_\pi. $$
\end{lemma}

\begin{proof} We have
\begin{align*}
    2 \langle \hat \Delta \rho \cdot \nabla f,\nabla f \rangle_\pi &= 
    \sum_{x,y \in X} \hat \Delta\rho(x,y) (f(y)-f(x))^2 Q(x,y) \pi(x) \\ 
    &= 
    \sum_{x,y \in X} \partial_1\theta(\rho_x,\rho_y) \Delta \rho(x) (f(y)-f(x))^2 Q(x,y) \pi(x) \\ 
    &+ \sum_{x,y \in X} \underbrace{\partial_2\theta(\rho_x,\rho_y)}_{=\partial_1\theta(\rho_y,\rho_x)} \Delta \rho(y) (f(x)-f(y))^2 \underbrace{Q(x,y) \pi(x)}_{=Q(y,x) \pi(y)} \\
    &= \sum_{x \in X} \Delta \rho(x) \underbrace{\sum_{y \in X} 2\partial_1\theta(\rho_x,\rho_y) (f(y)-f(x))^2 Q(x,y)}_{= 2 \Gamma_\rho f(x)}\pi(x)  \\
    &= 2 \langle \Delta \rho, \Gamma_\rho f \rangle_\pi =
    2 \langle \rho, \Delta (\Gamma_\rho f) \rangle_\pi.
\end{align*}
This shows \eqref{eq:cruc} and has the following implication:
\begin{align*}
\mathcal{B}_\rho(f) &= \frac{1}{2} \langle \hat \Delta \rho \cdot \nabla f, \nabla f \rangle_\pi - \langle \hat \rho \cdot \nabla f, \nabla \Delta f \rangle_\pi \\
&\stackrel{\mathclap{\eqref{eq:cruc}}}{=} \frac{1}{2} \langle \rho, \Delta(\Gamma_\rho f) \rangle_\pi - \langle \hat \rho \cdot \nabla f, \nabla \Delta f \rangle_\pi \\
&\stackrel{\mathclap{\eqref{eq:nablaGammarho}}}{=} \frac{1}{2} \langle \rho, \Delta(\Gamma_\rho f) \rangle_\pi - \langle \rho, \Gamma_\rho(f, \Delta f) \rangle_\pi = \langle \rho,\Gamma_{2,\rho}f \rangle_\pi.
\end{align*}
\end{proof}

\subsection{Curvature notions}
\label{subsec:curvnot}

Let $\theta$ be the logarithmic mean $\theta=\theta_{log}$. Recall the Bochner-type formula
$$ \mathcal{B}_\rho(f) \ge K \mathcal{A}_\rho(f) $$
for all $\rho \in \mathcal{P}_*(X)$ and $f \in \R^X$ from Theorem \ref{thm:EM4.5}(c). By Lemmas \ref{lem:Arhoprop} and \ref{lem:Brhoprop},
this inequality can be equivalently rewritten as
$$ \langle \rho, \Gamma_{2,\rho} f - K \Gamma_\rho f \rangle_\pi \ge 0. $$
This reformulation of a lower bound on the entropic curvature in terms of the modified $\Gamma$-operators motivates the following general curvature notions.

\begin{definition}[$CD_\theta(K,n)$, $K_n(X)$ and $K_n(\rho)$] \label{def:curvnotions}
Let $(X,Q,\pi)$ be a finite Markov chain with a mean $\theta$. 
For $K\in \mathbb{R}$ and $n\in (0,\infty]$, we say that \emph{$X$ satisfies $CD_\theta(K,n)$}, if we have, for all $\rho \in (I_\theta)^X$,
\begin{equation} \label{eq:CHthetaKn} 
\langle \rho,\Gamma_{2,\rho} f - \frac{1}{n} (\Delta f)^2 - K \Gamma_{\rho} f  \rangle_\pi \ge 0 \quad \text{for all $f\in \mathbb{R}^X$},
\end{equation}
where the dependence of $\theta$ is implicitly in the definition of $\Gamma_\rho$ and $\Gamma_{2,\rho}$.

If \eqref{eq:CHthetaKn} holds for a special measure $\rho \in (I_\theta)^X$, we say that \emph{the measure $\rho$ satisfies $CD_\theta(K,n)$}.

Furthermore, the $n$-dimensional curvatures for a fixed parameter $n\in (0,\infty]$ are defined as the following optimal curvature bounds:
$$ K_n(X) := \sup \{K \in \mathbb{R}:\ X \text{ satisfies } CD_\theta(K,n)\} $$
and
$$ K_n(\rho) := \sup \{K \in \mathbb{R}:\ \rho \text{ satisfies } CD_\theta(K,n)\}. $$
\end{definition}

\begin{remark} \label{rem:commentscurv} Let us provide various comments about Definition \ref{def:curvnotions}.
    \begin{itemize}
    \item[(a)] It follows immediately from the definition that
    \begin{equation} \label{eq:KnXKnrho} 
    K_n(X) = \inf \{ K_n(\rho):\, \rho \in (I_\theta)^X \}, 
    \end{equation}
    and, for any $c > 0$ and $\rho\in (I_\theta)^X$,
    $$ K_n(c \rho) = K_n(\rho). $$
    \item[(b)] An alternative equivalent description of $K_n(\rho)$ for a fixed measure $\rho \in (I_\theta)^X$ is the following:
    $$ K_n(\rho) = \inf \left\{ \frac{\langle \rho, \Gamma_{2,\rho} f - \frac{1}{n} (\Delta f)^2 \rangle_\pi}{\langle \rho, \Gamma_{\rho} f  \rangle_\pi}:\, \text{$f \in \R^X$ with $\langle \rho, \Gamma_\rho f \rangle_\pi > 0$} \right\}. $$
    \item[(c)] $X$ has entropic curvature bounded below by $K \in \R$ if and only if $X$ satisfies $CD_{ent}(K,\infty) = CD_\theta(K,\infty)$ for the logarithmic mean $\theta= \theta_{log}$.
    \item[(d)] We have $K_n(x) \ge K$ for the Bakry-\'Emery curvature at vertex $x \in X$ (as defined in Definition \ref{def:BEcurv}) if and only if the Dirac measure $\delta_x$, given in \eqref{eq:Diracx}, satisfies $CD_a(K,n) = CD_\theta(K,n)$ for the arithmetic mean $\theta = \theta_a$.
    \item[(e)] By decoupling the appearance of the measure $\rho$ in both arguments of the inner product in \eqref{eq:CHthetaKn}, one could also consider curvature-dimension conditions for pairs of measures $\eta \in [0,\infty)^X$ and $\rho \in (I_\theta)^X$ of the type
    $$ \langle \eta,\Gamma_{2,\rho} f - \frac{1}{n} (\Delta f)^2 - K \Gamma_{\rho} f  \rangle_\pi \ge 0 \quad \text{for all $f\in \mathbb{R}^X$.} $$
    However, we will not pursue this further in this paper.
    \end{itemize}
\end{remark}

It is interesting that the smallest positive eigenvalue of $-\Delta$ can be also viewed as a curvature value, by choosing the equilibrium measure $\rho = \textbf{1}_X$.

\begin{proposition}[$\lambda_1$ as ``curvature'' value]
\label{prop: lambda_1}
    Let $(X,Q,\pi)$ be a finite Markov chain and $\theta$ be an arbitrary mean. Let $\lambda_1(X)$ denote the smallest positive eigenvalue of $-\Delta$. Then we have 
    \begin{equation} \label{eq:lambda1curv}
    \lambda_1(X) = K_\infty(\textbf{1}_X),
    \end{equation}
    and, consequently, the following Lichnerowicz-type inequality
    \begin{equation} \label{eq:Lichineq} \lambda_1(X) \ge K_\infty(X).
    \end{equation}
\end{proposition}

    \begin{proof}
    The following variational description of $\lambda_1(X)$ is a bit different from the standard Rayleigh quotient but it can be proved similarly:
    \begin{align*}
    \lambda_1(X) = \inf_{f} \frac{\langle -\Delta f, -\Delta f \rangle_\pi}{\langle f, -\Delta f \rangle_\pi} = \inf_{f} \frac{-\langle \nabla  f, \nabla (\Delta f) \rangle_\pi}{\langle \nabla f, \nabla f \rangle_\pi} = \inf_{f} \frac{-\langle \textbf{1}_X, \Gamma(f,\Delta f) \rangle_\pi}{\langle \textbf{1}_X, \Gamma f  \rangle_\pi}.
    \end{align*}
    Here the infima are taken over all functions $f$ satisfying $\Gamma f \not\equiv 0$.
In the case $\rho=\textbf{1}_X$, we have $\Gamma_\rho=\Gamma$ and $\Gamma_{2,\rho}=\Gamma_2$, and we have from 
Remark \ref{rem:commentscurv}(b)
that
\begin{align*}
    K_\infty(\textbf{1}_X) = \inf_{f} \frac{\langle \textbf{1}_X, \Gamma_{2} f\rangle_\pi}{\langle \textbf{1}_X, \Gamma f \rangle_\pi}= \inf_{f} \frac{-\langle \textbf{1}_X, \Gamma(f,\Delta f) \rangle_\pi}{\langle \textbf{1}_X, \Gamma f \rangle_\pi}=\lambda_1(X),
\end{align*}
where the second equality is obtained by using 
$\Gamma_2 f =\frac{1}{2}\Delta(\Gamma f)-\Gamma(f,\Delta f)$ 
and applying the Divergence Theorem to conclude $\langle \textbf{1}_X, \Delta(\Gamma f)\rangle_\pi=0$.
This proves \eqref{eq:lambda1curv}. For the proof of \eqref{eq:Lichineq}, we apply \eqref{eq:KnXKnrho}.
\end{proof}

\begin{remark} \label{rem:Lichsharp}
    Markov chains, for which \eqref{eq:Lichineq} holds with equality, are called \emph{Lichnerowicz sharp}. As mentioned in the Introduction, examples of Lichnerowicz sharp Markov chains are simple random walks on $N$-dimensional hypercubes in both the case of 
    Bakry-\'Emery curvature and entropic curvature, since they satisfy $\lambda_1(Q^N)=\frac{2}{N}=K_\infty(Q^N)$
    in these cases. 
\end{remark}

Let us finish this section with the following simple interpolation result.

\begin{remark}[Trade-off between curvature and dimension] Note that if a finite Markov chain $(X,Q,\pi)$ with a mean $\theta$ satisfies both $CD_{\theta}(K_1,n_1)$ and $CD_{\theta}(K_2,n_2)$, then it also satisfies $$CD_{\theta}\left(\alpha K_1 + \beta K_2,\frac{n_1n_2}{\alpha n_2+\beta n_1} \right)$$ for $\alpha,\beta \ge 0$ and $\alpha+\beta=1$. This follows immediately by taking a convex combination of the inequalities \eqref{eq:CHthetaKn} with the parameter pairs $(K_1,n_1)$ and $(K_2,n_2)$. 

In particular, Markov chains satisfying both $CD_\theta(K,\infty)$ and $CD_\theta(0,n)$, do also satisfy $CD_\theta(\alpha K, \frac{n}{1-\alpha})$ for any choice of $\alpha \in (0,1)$.
\end{remark}

\section{Optimal measures for Bakry-\'Emery curvature}
\label{sec:optmeas}

In this section, we restrict our considerations entirely to finite Markov chains $(X,Q,\pi)$ with the arithmetic mean $\theta_a$. Consequently, our carr\'e du champ operators 
are the standard Bakry-\'Emery ones, that is $\Gamma_\rho = \Gamma$ and $\Gamma_{2,\rho} = \Gamma_2$ for any choice of measure $\rho \in [0,\infty)^X$.
Recall that the classical Bakry-\'Emery curvature $K_n(x)$ at a vertex $x \in X$, given by
$$
K_n(x) = \inf_{\Gamma f(x) > 0} \frac{\Gamma_2 f (x) - \frac{1}{n} (\Delta f(x))^2}{\Gamma f(x)},
$$
agrees with the supremum of all $K \in \R$ such that
\begin{equation} \label{eq:deltaxcurv}
\langle \delta_x,\Gamma_{2} f - \frac{1}{n} (\Delta f)^2 - K \Gamma f \rangle_\pi \ge 0 \quad \text{for all $f \in \mathbb{R}^X$,}
\end{equation}
where $\delta_x$ is the Dirac measure defined in \eqref{eq:Diracx}. This viewpoint allows us to consider the Bakry-\'Emery curvature $K_n(\rho)$ for general measures $\rho$, by replacing $\delta_x$ by $\rho$ in \eqref{eq:deltaxcurv}, and taking the supremum over all $K \in \R$.
The $n$-dimensional curvature of the whole space $X$ is then given by (see Remark \ref{rem:commentscurv}(a))
\begin{equation} \label{eq:KnXKnrhoBE}
K_n(X) = \inf \{ K_n(\rho):\, \rho \in [0,\infty)^X \}. 
\end{equation}
Even though $K_n(X)$ is defined as the infimum of the Bakry-\'Emery curvatures over all measures $\rho$, it suffices to only consider the Bakry-\'Emery curvatures of the Dirac measures $\delta_x$, $x \in X$, as the following lemma states.

\begin{lemma} \label{lem:Knrels}
Let $n \in (0,\infty]$. For any $\rho\in [0,\infty)^X$, we have
        $$ K_n(\rho) \ge \min_{x \in \supp(\rho)} K_n(x).$$
In particular, 
        $$ K_n(X) = \min_{x \in X} K_n(x). $$
\end{lemma}

\begin{proof}
Fix a measure $\rho\in [0,\infty)^X$. Denote $K:=\min_{x \in \supp(\rho)} K_n(x)$. Suppose for the sake of contradiction that $K_n(\rho) < K$. By the definition of $K_n(\rho)$, there exists $f\in \mathbb{R}^X$ such that $\langle \rho,\Gamma_{2}(f) - \frac{1}{n} (\Delta f)^2 - K \Gamma f \rangle_\pi < 0$.
By decomposing $\rho=\sum_x \rho(x)\textbf{1}_x$ with $\textbf{1}_x := \pi(x) \delta_x$, we derive
\[
0>\langle \rho,\Gamma_{2} f - \frac{1}{n} (\Delta f)^2 - K \Gamma f \rangle_\pi = \sum_x \rho_x\langle \textbf{1}_x,\Gamma_{2} f - \frac{1}{n} (\Delta f)^2 - K \Gamma f  \rangle_\pi \ge 0,
\] 
which is a contradiction. Thus we proved the first statement of the lemma. The second statement follows immediately from the first statement, \eqref{eq:KnXKnrhoBE} and $K_n(x) = K_n(\delta_x)$. 
\end{proof}

Lemma \ref{lem:Knrels} motivates the following definition of an optimal measure.

\begin{definition}[Optimal measure]
    Let $(X,Q,\pi)$ be a finite Markov chain.
    A measure $\rho \in [0,\infty)^X$ is \emph{optimal for dimension $n \in (0,\infty]$}, if there exists $f_0 \in \mathbb{R}^X$ with $\Gamma f_0>0$ on 
    $\supp \rho$, such that 
    \begin{equation} \label{eq:opt-measure-BE}
    \langle \rho,\Gamma_{2} f_0 - \frac{1}{n} (\Delta f_0)^2 - K_n(X) \Gamma f_0 \rangle_\pi =0.
    \end{equation}
\end{definition}

Note that we have $K_n(\rho) = K_n(X)$ for any optimal measure $\rho$, but not necessarily vice versa (see Example \ref{ex:nonopt} below). 
Moreover, it follows from Lemma \ref{lem:Knrels} that any Dirac measure $\delta_x$ with $K_n(x) = K_n(X)$ is an optimal measure of $X$.
Another example of an optimal measure is the equilibrium measure in the case of a Lichnerowicz sharp Markov chain (see Remark \ref{rem:Lichsharp}).

\begin{proposition} \label{prop:Lichopt}
    Let $(X,Q,\pi)$ be a finite Markov chain. Then $X$ is Lichnerowicz sharp if and only if $\textbf{1}_X$ is an optimal measure for dimension $n=\infty$. 
\end{proposition}

\begin{proof}
In view of Proposition \ref{prop: lambda_1}, we know that
  \begin{equation} \label{eq: Lich_sharp}
      \lambda_1(X) = K_\infty(\textbf{1}_X) \ge K_\infty(X).
  \end{equation}
  Thus if $\textbf{1}_X$ is an optimal measure for dimension $n=\infty$, then the inequality in \eqref{eq: Lich_sharp} holds with equality, which means $X$ is Lichnerowicz sharp: $\lambda_1(X)=K_\infty(X)$.

  Conversely, suppose that $X$ is Lichnerowicz sharp. We claim that an eigenfunction $f_0\in \R^X$ such that $\Delta f_0=-\lambda_1(X) f_0$ must satisfy 
  \begin{equation} \label{eq: Lichsharp_proof0}
      \langle \textbf{1}_X, \Gamma_2 f_0-K_\infty(X) \Gamma f_0 \rangle_\pi=0,
  \end{equation}
  and more importantly, $\Gamma f_0>0$ everywhere. In fact, we will prove that $\Gamma f_0=c$ for some constant $c>0$.

  To prove this claim, we consider at any vertex $x\in X$,
  \begin{align} \label{eq: Lichsharp_proof1}
      \Gamma_2 f_0(x) 
      &= \frac{1}{2} \Delta (\Gamma f_0)(x)- \Gamma (f_0,\Delta f_0)(x) \nonumber \\
      &= \frac{1}{2} \Delta (\Gamma f_0)(x)+ \lambda_1(X)\Gamma f_0 (x).
  \end{align}
  On the other hand, by the definition 
  $$ \Gamma_2 f(x) - K_\infty(\delta_x)\Gamma f(x) = \langle \delta_x, \Gamma_2f - K_\infty(\delta_x)\Gamma f \rangle_\pi \ge 0 $$
  for all $f \in \R^X$, we have
  \begin{align} \label{eq: Lichsharp_proof2}
      \Gamma_2 f_0(x) \ge K_\infty(\delta_x) \Gamma f_0(x) \ge K_\infty(X) \Gamma f_0(x).
  \end{align}
  Combining \eqref{eq: Lichsharp_proof1}, \eqref{eq: Lichsharp_proof2} and the Lichnerowicz sharpness assumption, we deduce that $\Delta (\Gamma f_0)(x) \ge 0$ for every vertex $x$. However, since 
  $$ 0 = \langle \Delta \textbf{1}_X, \Gamma f_0 \rangle_\pi = \langle \textbf{1}_X, \Delta (\Gamma f_0)\rangle_\pi, $$
  we conclude that $\Delta (\Gamma f_0)(x)=0$. It follows that the inequalities in \eqref{eq: Lichsharp_proof2} must hold with equality. Thus $\Gamma_2 f_0-K_\infty(X) \Gamma f_0=0$ everywhere, which then yields \eqref{eq: Lichsharp_proof0}. The fact that $\Delta (\Gamma f_0)=0$ everywhere also implies that $\Gamma f_0=c$ for some constant $c\ge 0$, since $X$ is assumed to be irreducible and hence its underlying graph is connected. Lastly, $c\not=0$; otherwise, it would imply that $f_0$ is constant and that $\lambda_1(X)=0$, contradicting to the assumption that $\lambda_1(X)$ is the smallest positive eigenvalue of $-\Delta$.
\end{proof}

\begin{proposition}
\label{prop:optimalmon}
Let $\rho \in [0,\infty)^X$ be an optimal measure for dimension $n \in (0,\infty]$ of a finite Markov chain $(X,Q,\pi)$.
\begin{itemize}
    \item[(a)] If $\rho_0\le \rho$, then $\rho_0$ is also optimal for dimension $n$.
    \item[(b)] For every $x \in {\rm supp}(\rho)$, we have $K_n(x)=K_n(X)$.
\end{itemize}
\end{proposition}

\begin{proof}
    We first prove (a). Let $f_0 \in \R^X$ be a function with
    $\Gamma f_0 > 0$ on $\supp \rho$, 
    satisfying \eqref{eq:opt-measure-BE}.
    Writing $\rho=\rho_0+\rho_1$, we have
    \[
    \sum_{i=0}^{1} \langle \rho_i,\Gamma_{2} f_0 - \frac{1}{n} (\Delta f_0)^2 - K_n(X) \Gamma f_0 \rangle_\pi =0,
    \]
    where each of the inner products has non-negative value because $K_n(X)$ is the greatest lower curvature bound. It follows that $$ \langle \rho_0,\Gamma_{2} f_0 - \frac{1}{n} (\Delta f_0)^2 - K_n(X) \Gamma f_0 \rangle_\pi =0. $$
    Moreover, $\Gamma f_0 >0$ on $\supp(\rho_0)$ since $\supp(\rho_0) \subset \supp(\rho)$. Thus $\rho_0$ is optimal, proving (a). 
    
    For any $x\in \supp(\rho)$, choosing $\rho_0=\rho(x)\textbf{1}_x$ immediately yields (b).
\end{proof}

It follows from part (a) of Proposition \ref{prop:optimalmon}, that we can define optimality for a set $A \subset X$ where $A={\rm supp}(\rho)$ for an optimal $\rho$. Moreover, if $A \subset X$ is optimal, then all $A_0 \subset A$ are also optimal. In particular, the family of optimal sets of a Markov chain $X$ form a simplicial complex whose $0$-cells are precisely the elements of $X_0:=\{x\in X:\ K_n(x)=K_n(X)\}$, that is the set of all vertices with minimal curvature. If $A\subset X$ is optimal, then $A \subset X_0$ by part (b) of Proposition \ref{prop:optimalmon}. However, $X_0$ itself does not need to be optimal. Examples are Markov chains with constant Bakry-\'Emery curvature on its vertices (e.g., vertex transitive graphs) which are not Lichnerowicz sharp (see Remark \ref{rem:Lichsharp}).  Examples for which $X_0$ is  not optimal are cycles of length $\ge 5$, as discussed in the following example. It would be interesting to study the simplicial complexes of optimal sets for other vertex transitive graphs.

\begin{example}[Optimal sets of a cycle]
Consider a Markov chain $(X,Q,\pi)$ associated to a simple random walk on a cycle of length $n \ge 5$, that is, $X=\mathbb{Z}/n\mathbb{Z}$ and $Q(x,x-1)=Q(x,x+1)=\frac{1}{2}$ for all $x\in X$, with indices taken modulo $n$. For example, we can write $X = \{-2,-1,0,1,2,\dots,n-3\}$.
Next we explicitly compute $$K_\infty(x)= \inf_{\Gamma f(x) > 0} \frac{\Gamma_2 f (x)}{\Gamma f(x)} $$ 
and find its corresponding minimizers $f$.

For simplicity, we write $a_i=f(x+i)-f(x+i-1)$ for all $i\in\{-1,0,1,2\}$. Then
\begin{align*}
    \Delta f(x) &= \frac{1}{2}(a_1-a_0),\\
    \Gamma f(x) &= \frac{1}{4}(a_1^2+a_0^2),\\
    \Gamma(f,\Delta f)(x) &= \frac{1}{2}\sum_y Q_{xy}(f(y)-f(x))(\Delta f(y)-\Delta f(x)) \\
    &= \frac{1}{8}(a_1(a_2-a_1-a_1+a_0)-a_0(a_0-a_{-1}-a_1+a_0)) ,\\
    \Delta(\Gamma f)(x) &= \sum_y Q_{xy}(\Gamma f(y)-\Gamma f(x)) = \frac{1}{8}(a_2^2-a_1^2-a_0^2+a_{-1}^2) ,\\
    \Gamma_2 f(x) &= \frac{1}{16}(a_2^2+3a_1^2+3a_0^2+a_{-1}^2-2a_1a_2-2a_1a_0-2a_0a_{-1}-2a_1a_0) \\
    &= \frac{1}{16}\left((a_2-a_1)^2+2(a_1-a_0)^2+(a_0-a_{-1})^2\right).
\end{align*}
The expression for $\Gamma_2 f(x)$ shows that $K_\infty(x)=0$ and  
a corresponding minimizer $f_0$ must satisfy the condition that
\begin{itemize}
    \item[($*$)] $f_0(x+2),f_0(x+1),f_0(x),f_0(x-1),f_0(x-2)$ forms an arithmetic progression which is non-constant due to the condition $\Gamma f_0(x) \neq 0$.
\end{itemize}
Then a set $A \subset X$ is optimal if and only if there exists a non-constant function $f_0$ satisfying the condition ($*$) for all $x\in A$. With this function $f_0$, we have
\begin{equation} \label{eq:AmM}
A\subset X-\{m-1,m,m+1\}-\{M-1,M,M+1\}, 
\end{equation}
where $m$ and $M$ are the vertices where the non-constant function $f_0$ is minimal and maximal, respectively, since none of $x=m-1,m,m+1,M-1,M,M+1$ can satisfy condition ($*$), by the extremality property of $f_0$ at $m$ and $M$. This shows that, for any optimal set $A \subset X$, there are distinct vertices $m,M \in X$ such that \eqref{eq:AmM} holds. 

Conversely, any choice of $m,M \in X$, $m \neq M$, yields a corresponding optimal set $A = X-\{m-1,m,m+1\}-\{M-1,M,M+1\}$ by choosing a non-constant function $f_0$ with (arbitrary) minimum $f_0(m)$ at $m$ and (arbitrary) maximum $f_0(M) > f_0(m)$ at $M$ and interpolating linearly between the vertices $m$ and $M$ along both connecting paths in the cycle. Then $f_0$ satisfies property ($*$) 
at all vertices $x \in A$.
In particular, if $|M-m| = 1$, this yields optimal sets $A$ of size $n-4$.
Therefore, the family of all optimal sets forms an $(n-5)$-dimensional simplicial complex (whose simplices of maximal dimension $(n-5)$ are all of the form $X-\{m-1,m,m+1,m+2\}$).
\end{example}

\begin{example}[Optimal sets of a hypercube] Consider the random walk on the hypercube $Q^N$. Since this Markov chain is Lichnerowicz sharp, the whole vertex set $Q^N$ is optimal and the associated simplicial complex consists of all subsets of $Q^N$. The same holds for any Lichnerowicz sharp finite Markov chain.
\end{example}

We finish this section with considerations about unions of optimal sets.

\begin{proposition}
    Let $(X,Q,\pi)$ be a finite Markov chain, $n \in (0,\infty]$
	and $A_0,A_1 \subset X$ be optimal sets for dimension $n$. If $d(A_0,A_1) \ge 5$, then $A_0\cup A_1$ is also optimal for dimension $n$.
\end{proposition}

\begin{proof}
	Set $K:=K_n(X)$. For $i\in\{0,1\}$, let $f_i\in \R^X$ be a function with $\Gamma f_i>0$ on $A_i$ with
	$$
	\langle \textbf{1}_{A_i},\Gamma_{2} f_i - \frac{1}{n} (\Delta f_i)^2 - K \Gamma f_i \rangle_\pi =0.
	$$ 
	Choose a function $f\in \R^X$ such that, for $i\in\{0,1\}$, it satisfies $f(x)=f_i(x)$ for all $x$ with $d(x,A_i)\le 2$. This is always possible under the assumption $d(A_0,A_1) \ge 5$. 
     On each of the sets $A_i$, we have $\Delta f=\Delta f_i$,  $\Gamma f=\Gamma f_i > 0$ and $\Gamma_{2} f=\Gamma_{2} f_i$, and therefore
	$$
	\langle \textbf{1}_{A_0\cup A_1},\Gamma_{2} f - \frac{1}{n} (\Delta f)^2 - K \Gamma f \rangle_\pi =0,
	$$ 
	which shows that $A_0\cup A_1$ is optimal.
\end{proof}

\begin{example}[A non-optimal measure with $K_n(\rho) = K_n(X)$] \label{ex:nonopt}
Let $x,y \in X$ be two distinct vertices with distance $d(x,y) \ge 5$ and
$$ K_n(x) = K_n(X) < K_n(y). $$
Consider the measure $\rho = {\textbf{1}}_x + {\textbf{1}}_y$ and a minimizer $f_0$ satisfying 
$$ \langle \textbf{1}_x, \Gamma_2 f_0 - \frac{1}{n}(\Delta f_0)^2 - K_n(X) \Gamma f_0 \rangle_\pi = 0, $$
and $f_0 \equiv 0$ on $X \setminus B_2(x)$. Since $B_2(x) \cap B_2(y) = \emptyset$, we have $\Gamma_2 f_0(y) = \Delta f_0(y) = \Gamma f_0(y) = 0$ and therefore
$$ \langle \textbf{1}_y, \Gamma_2 f_0 - \frac{1}{n}(\Delta f_0)^2 - K_n(X) \Gamma f_0 \rangle_\pi = 0, $$
which implies $K_n(\rho) = K_n(X)$. However, $\rho$ is not an optimal measure by Proposition \ref{prop:optimalmon}(b), since $K_n(y) > K_n(x)$.
\end{example}

\section{A gradient estimate}
\label{sec:gradest}

The aim of this section is to provide an equivalent formulation of the $CD_\theta(K,n)$-property via a gradient estimate. We follow closely \cite[Theorem 3.1]{EF-18} and its proof,
and present it with the additional dimension parameter, for the readers' convenience. 

\begin{theorem}[Gradient estimate] \label{thm:gradest}
Let $(X,Q,\pi)$ be a finite Markov chain and $\theta$ an arbitrary mean. Let $K \in \mathbb{R}$ and $n\in (0,\infty]$. The following statements are equivalent.
\begin{enumerate}
    \item[(i)] $(X,Q,\pi)$ satisfies $CD_\theta(K,n)$, that is, for all $\rho \in (I_\theta)^X, f \in \mathbb R^X$,
    \[
    \mathcal B_\rho (f) \geq K \mathcal A_\rho (f) + \frac 1 n\langle \rho, (\Delta f)^2 \rangle_\pi,
    \]
    with operators $\mathcal{A}_\rho$ and $\mathcal{B}_\rho$ defined in \eqref{eq:Arhodef} and \eqref{eq:Brhodef}, respectively.
    \item[(ii)] For all $\rho \in (I_\theta)^X, f \in \mathbb R^X$, $t \ge 0$,
    \[
    e^{-2Kt} \mathcal A_{P_t \rho}(f) - \mathcal A_\rho(P_t f) \geq \frac{1-e^{-2Kt}}{Kn}\langle \rho,(\Delta P_t f)^2 \rangle_\pi,
    \]
    where $P_t: L^2(X,\pi) \to L^2(X,\pi)$ is the heat semigroup operator.
\end{enumerate}
\end{theorem}

\begin{proof}
We first prove $(i) \Rightarrow (ii)$. Let $\rho_s = P_s \rho$ and $f_s = P_s f$.
    We consider
    \[
    G(s) := e^{-2Ks}\mathcal A_{P_s \rho} (P_{t-s} f) = 
    e^{-2Ks} \langle \hat \rho_s \cdot \nabla f_{t-s},\nabla f_{t-s}\rangle_\pi.
    \]
Using the identity $\partial_s \hat \rho_s = \hat \Delta \rho_s$ (see \eqref{eq:hatderid2} and \eqref{eq:hatDelta} for $\hat \rho$ and $\hat \Delta$), we obtain
\begin{align*}
\partial_s G(s) &= -2KG(s) + e^{-2Ks} \left(2\mathcal B_{\rho_s} (f_{t-s}) \right)
\\&\geq -2KG(s)  +  e^{-2Ks} \left(2K \mathcal A_{\rho_s} (f_{t-s}) + \frac{2}{n}\langle \rho_s, (\Delta f_{t-s})^2 \rangle_\pi \right)    \\
&\geq \frac{2}{n} e^{-2Ks} \langle \rho, (\Delta P_{t} f)^2 \rangle_\pi,
\end{align*}
where we applied the $CD_\theta(K,n)$ inequality in the first estimate and $P_s(g^2) \ge (P_s g)^2$ (which follows via Cauchy Schwarz) in the second estimate.
Integrating from $s=0$ to $t$ proves $(ii)$.  

For the reverse implication, consider
$$ F(t) = e^{-2Kt} \mathcal{A}_{P_t\rho}(f) - \mathcal{A}_\rho(P_tf) - \frac{1-e^{-2Kt}}{Kn} \langle \rho, (\Delta P_t f)^2 \rangle_\pi, $$
and note that $F(0) = 0$ and $F(t) \ge 0$ for $t \ge 0$, by (ii). This implies $F'(0) \ge 0$ with
\begin{align*}
  F'(0) &= - 2 K \mathcal{A}_\rho(f) + \partial_t\vert_{t=0} \left( \mathcal{A}_{\rho_t}(f) - \mathcal{A}_\rho(f_t) \right) - \frac{2}{n} \langle \rho,(\Delta f)^2\rangle_\pi \\
  &= - 2 K \mathcal{A}_\rho(f) + \langle \partial_t\vert_{t=0} \hat \rho_t \cdot \nabla f,\nabla f \rangle_\pi - 2 \langle \hat \rho \cdot \nabla f, \nabla (\Delta f)\rangle_\pi - \frac{2}{n} \langle \rho,(\Delta f)^2\rangle_\pi \\
  &= - 2 K \mathcal{A}_\rho(f) + 2 \mathcal{B}_\rho(f) - \frac{2}{n} \langle \rho,(\Delta f)^2\rangle_\pi.
\end{align*}
% The reverse implication follows by taking derivative of $(ii)$ at $t=0$. This finishes the proof.  
\end{proof}

One important application of the gradient estimate characterization (Theorem \ref{thm:gradest}) is to derive the following reverse Poincar\'e inequality (see \cite[Theorem 3.5]{EF-18}).
\begin{theorem} [Reverse Poincar\'e inequality] \label{thm:revPoinc}
    If a Markov chain $(X,Q,\pi)$ with a mean $\theta \le \theta_a$ satisfies $CD_\theta(K,n)$, then
    \begin{equation} \label{eq:reverse_poincare}
    \langle f^2, P_t\rho\rangle_\pi-\langle (P_t f)^2, \rho \rangle \ge \frac{e^{2Kt}-1}{K} \mathcal{A}_\rho (P_t f) +\frac{1}{Kn}\left( \frac{e^{2Kt}-1}{K}-2t \right) \langle \rho, (\Delta P_t f)^2 \rangle_\pi
    \end{equation}
In particular, $CD_\theta(0,\infty)$ implies
\begin{equation} \label{eq:rp_zeroinf}
    \langle f^2, P_t\rho\rangle_\pi \ge 2t\mathcal{A}_{\rho} (P_t f)
\end{equation}
\end{theorem}

\begin{proof}
We consider $F(s) := \langle (P_{t-s}f)^2, P_{s}\rho\rangle_\pi$ and compute its derivative
\begin{align*}
    \partial_s F(s) &= \langle 2(P_{t-s}f) (\partial_s P_{t-s}f) , P_s\rho \rangle_\pi + 
    \langle (P_{t-s}f)^2, \Delta P_s\rho \rangle_\pi \\
    &= \langle -2(P_{t-s}f) (\Delta P_{t-s}f)+\Delta (P_{t-s}f)^2, P_s\rho\rangle_\pi \\
    &= \langle 2\Gamma (P_{t-s}f), P_s\rho\rangle_\pi \\
    &\ge 2\mathcal{A}_{P_{s}\rho}(P_{t-s}f)\\
    &\ge 2e^{2Ks} \mathcal{A}_\rho(P_tf) + \frac{2(e^{2Ks}-1)}{Kn}\langle \rho, (\Delta P_t f)^2\rangle_\pi,
\end{align*}
where inequalities are due to \eqref{eq:Agammaest} in Lemma \ref{lem:Arhoprop}
and Theorem \ref{thm:gradest}(ii), respectively. Integrating from $s=0$ to $t$ yields the inequality \eqref{eq:reverse_poincare}. Choosing $n=\infty$ and letting $K\to 0$ simplifies the inequality to \eqref{eq:rp_zeroinf} as desired.
\end{proof}

\begin{corollary} \label{cor:rp_zeroinf}
    Let $(X,Q,\pi)$ satisfy non-negative curvature condition $CD_\theta(0,\infty)$ for some mean $\theta \le \theta_a$. Then we have for all $t\ge 0$ and all $f\in \mathbb{R}^X$,
\begin{align} \label{eq:MG}
\|\nabla P_tf\|_\infty \le \frac{\|f\|_\infty}{\sqrt{tQ_{\min}}},
\end{align}
or equivalently, for all $t\ge 0$ and all $\mu,\nu \in \mathcal{P}(X)$,
\begin{align} \label{eq:MG-equiv}
\|P_t\mu-P_t\nu\|_1 \le \frac{W_1(\mu,\nu)}{\sqrt{tQ_{\min}}}.
\end{align}
Here,
$$ Q_{\min} := \min_{x\sim y} Q(x,y), $$
and $W_1(\mu,\nu)$ denotes the $1$-Wasserstein distance between measures with densities $\mu$ and $\nu$, that is, $W_1(\mu,\nu):= \sup\limits_{\|\nabla g\|_\infty \le 1} \langle \mu-\nu, g \rangle_\pi$ by the reformulation via Kantorovich duality .
\end{corollary}

\begin{proof}
We localize the reverse Poincar\'e inequality \eqref{eq:rp_zeroinf}, that is,
\begin{equation*}
    \langle f^2, P_t\rho\rangle_\pi \ge 2t\|\nabla P_tf\|_\rho^2
\end{equation*}
by choosing $\rho= \textbf{1}_x +\textbf{1}_y$ for arbitrary $x,y\in V$ such that $x\sim y$. Consequently, we obtain the the following estimates for the left-hand-side and right-hand side of the above inequality:
    \begin{align*}
    \langle f^2, P_t\rho\rangle_\pi \le \|f\|_\infty^2 \langle 1, P_t\rho \rangle_\pi = \|f\|_\infty^2 \langle P_t1, \rho \rangle_\pi = \|f\|_\infty^2(\pi(x)+\pi(y)),
    \end{align*}
    and
    \begin{align*}
    2t\|\nabla P_tf\|_\rho^2 &\ge t \left(P_tf(y)-P_tf(x)\right)^2(Q(x,y)\pi(x)+Q(y,x)\pi(y)) \\
    &\ge t \left(P_tf(y)-P_tf(x)\right)^2Q_{\min}(\pi(x)+\pi(y)).
    \end{align*}
Combining both inequalities above yields $t\left(P_tf(y)-P_tf(x)\right)^2 \le \frac{\|f\|_\infty^2}{Q_{\min}}$ for any $x\sim y$, which implies \eqref{eq:MG}. The equivalence between \eqref{eq:MG} and \eqref{eq:MG-equiv} is proved in \cite[Theorem 3.1.1]{Mu-23}. In short, the implication \eqref{eq:MG}$\Rightarrow$\eqref{eq:MG-equiv} is obtained by choosing $f={\rm sgn(P_t\mu-P_t\nu)}$, and the implication \eqref{eq:MG-equiv}$\Rightarrow$\eqref{eq:MG} is obtained by choosing $\mu=\frac{1}{\pi(x)}{\rm 1}_x$ and $\nu=\frac{1}{\pi(y)}{\rm 1}_y$ for arbitrary $x,y\in X$.
\end{proof}

\section{Diameter bounds}
\label{sec:diambound}

Before giving our own Bonnet-Myers type of diameter bound result, let us mention another such result from \cite{EF-18}.

\begin{theorem}[{\cite[Prop. 7.3]{EF-18}}] \label{thm:diambd-EF} Assume that an irreducible, reversible finite Markov chain $(X,Q,\pi)$ has entropic curvature bounded below by $K > 0$. Then
$$ {\rm{diam}}(X,d_\mathcal{W})  \le 2 \sqrt{\frac{-2 \log \pi_{\min}}{K}}, $$
where the distance $d_{\mathcal{W}}$ is defined by $d_{\mathcal{W}}(x,y):=\mathcal{W}(\frac{1}{\pi(x)}{\rm 1}_x,\frac{1}{\pi(y)}{\rm 1}_y)$ and $\pi_{\min}:=\min_{x \in X} \pi(x)$.
\end{theorem}
In view of \cite[Proposition 2.12 and Lemma 2.13]{EM-12}, the distance $d_\mathcal{W}$ and the combinatorial distance $d$ is related by
\[
 \sqrt{2}d(x,y) \le d_{\mathcal{W}}(x,y) \le \frac{1.57}{\sqrt{Q_{\min}}}d(x,y).
\] In particular, one can derive from Theorem \ref{thm:diambd-EF} the diameter bound \[{\rm{diam}}(X,d) \le 2 \sqrt{\frac{-\log \pi_{\min}}{K}}.\]

We like to mention that this curvature-diameter bound involves the value $\pi_{\min}$, which is related to the size $|X|$ of the Markov chain. In the case of the $N$-dimensional hypercube, it is sharp in the dimension parameter. The Bonnet-Myers type results given in this paper are not sharp in the dimension parameter, but they do not involve the overall size of the Markov chain and only local properties like the vertex degrees. Specifically, we present the diameter bound under the $CD_{ent}(K,\infty)$ assumption in Subsection \ref{subsec:diambd-infty} and the diameter bound under the 
$CD_{\theta}(K,n)$ assumption with $\theta \le \theta_a$ in Subsection \ref{subsec:diambd-finite}. The former result was already presented in the special setting of simple random walks in \cite[Theorem 1.3]{Ka-20} and in the general setting of reversible Markov chains in the PhD-thesis \cite{Ka-21}. Both results employ the characterization of the $CD_\theta$ condition
via a gradient estimate (Theorem \ref{thm:gradest}), and they make use of another distance function on $X$, introduced in \cite{LMP-18}, namely
\[
d_\Gamma (x,y) := \sup\{f(y)-f(x): \|\Gamma f\|_\infty \leq 1 \}.
\]
The distance functions $d_\Gamma$ and the combinatorial distance $d$ (\cite[Lemma 1.4]{LMP-18}) are related by
\begin{equation} \label{eq:ddGamma}
d(x,y) \le \sqrt{\frac{D}{2}} d_\Gamma(x,y) \le \frac{1}{\sqrt 2} d_\Gamma(x,y),
\end{equation}
where the second inequality follows from the definition of the maximal weighted vertex degree $D =\max_{x} D(x) \le 1$ with $D(x) = \sum_{y\not=x}Q(x,y)$ (see \eqref{eq:weighdeg}).

\subsection{Diameter bound under $CD_{ent}(K,\infty)$} \label{subsec:diambd-infty}

The diameter bound in this subsection is based on the following localized version of our gradient estimate.

\begin{corollary} [Local gradient estimate] \label{cor:local-gradest}
Assume that a
finite Markov chain $(X,Q,\pi)$ satisfies $CD_{ent}(K,\infty)$ with $K > 0$. Then for all $x\in X$ and all $\epsilon >0$,
    \begin{equation*} 
    \Gamma (P_tf)(x) \le \frac{e^{-2Kt}}{2\theta(1,\varepsilon)}\left( P_t(\Gamma f)(x) + \varepsilon \sum_{y\sim x} P_t(\Gamma f)(y) \frac{\pi(y)}{\pi(x)}\right).
    \end{equation*}
    In particular, for all $x\in X$,
    \[
		\Gamma(P_t f) (x) \le c\cdot e^{-2K t} \|P_t (\Gamma f)\|_{\infty},
	\]
	where $c:=\frac{D_\pi\log D_\pi}{D_\pi-1}$ and $D_\pi$ is the maximal $\pi$-vertex degree given by
 \begin{equation} \label{eq:degpi-defn}
     D_\pi:=\max_{x\in X} {D}_{\pi}(x) :=\max_{x\in X} \frac{1}{\pi(x)}\sum_{y\sim x}\pi(y).
 \end{equation}
    In the case $D_\pi=1$, we choose $c:=1$.
\end{corollary}

\begin{proof}
     By \eqref{eq:Agammaest} in Lemma \ref{lem:Arhoprop}
     and Theorem \ref{thm:gradest}, we have 
     \[
     \mathcal{A}_\rho(P_t f) \le e^{-2Kt} \mathcal{A}_{P_t\rho}(f) \le e^{-2Kt} \langle P_t\rho, \Gamma f\rangle_{\pi} = e^{-2Kt} \langle \rho, P_t(\Gamma f)\rangle_{\pi}.
     \]
     By choosing $\rho=\textbf{1}_x+\varepsilon\sum_{y\sim x}\textbf{1}_y$ for a fixed $x\in X$ and $\varepsilon>0$, we obtain
     \[
        \mathcal{A}_\rho(P_t f) \ge \theta(1,\varepsilon) \sum_{y\in X} (P_tf(y)-P_tf(x))^2Q(x,y)\pi(x) = 2\theta(1,\varepsilon) \Gamma (P_tf)(x) \pi(x),
     \]
     and
     \[
        \langle \rho, P_t(\Gamma f)\rangle_{\pi} = P_t(\Gamma f)(x)\pi(x) + \varepsilon \sum_{y\sim x} P_t(\Gamma f)(y)\pi(y).
     \]
     Therefore, we derive 
     \[
        2\theta(1,\varepsilon) \Gamma (P_tf)(x) \le e^{-2Kt} \left( P_t(\Gamma f)(x) + \varepsilon \sum_{y\sim x} P_t(\Gamma f)(y)\frac{\pi(y)}{\pi(x)} \right),
     \]
     as desired. The second inequality follows easily from \eqref{eq:logmeanformula} by taking $\varepsilon=\frac{1}{D_\pi}$.
\end{proof}

\begin{theorem} [Diameter bound under $CD_{ent}(K,\infty)$] \label{thm:diambd-infty}.
    Assume that a 
    finite Markov chain $(X,Q,\pi)$ satisfies $CD_{ent}(K,\infty)$ with $K>0$. Then 
    \[
    {\rm diam}(X,d_\Gamma) \le \frac{2}{K}\sqrt{\frac{2D_\pi \log D_\pi}{D_\pi-1}}.
    \]
    In particular,
    \[
    \diam(X,d) \le \frac{2}{K}\sqrt{\frac{D_\pi \log D_\pi}{D_\pi-1}}.
    \]
	Here, $D_\pi$ is the maximal $\pi$-vertex degree defined in \eqref{eq:degpi-defn} and $\frac{D_\pi \log D_\pi}{D_\pi-1} := 1$ if $D_\pi=1$.
\end{theorem}

\begin{proof}[Proof of Theorem \ref{thm:diambd-infty}]
	Let $x_0,y_0\in V$ such that $d_\Gamma(x_0,y_0)={\rm diam}(X,d_\Gamma)$ and let $f\in \mathbb{R}^X$ such that $\|\Gamma f\|_\infty \le 1$ and $f(y_0)-f(x_0)={\rm diam}(X,d_\Gamma)$.

    For all $x\in X$ and $g \in \R^X$, Cauchy-Schwarz implies $\left|\Delta g(x)\right|^2 \le 2\Gamma g(x)$ (see also \cite[p. 67]{LMP-18}). Therefore, we have together with Corollary \ref{cor:local-gradest},
	\[
	\left|\Delta P_t f(x)\right|^2 \le 2\Gamma(P_t f)(x)
	\le 2ce^{-2K t} \|P_t(\Gamma f)\|_{\infty} 
	\le 2ce^{-2K t} \|\Gamma f\|_{\infty}
	\le 2ce^{-2K t},
	\]
	where $c:=\frac{D_\pi\log D_\pi}{D_\pi-1}$.

For all $T>0$, we can derive 
\[|f(x)-P_Tf(x)| 
\le \int\limits_{0}^{T} \left|
\partial_t P_t f(x)\right|dt
= \int\limits_{0}^{T} \left|\Delta P_t f(x)\right|dt \le \int\limits_{0}^{T} \sqrt{2c}e^{-K t} dt \le \frac{\sqrt{2c}}{K}.
\]

Moreover, the fact that $\Gamma (P_t f)(x) \le c e^{-2K t}$ implies $|P_tf(x')-P_tf(x)| \rightarrow 0$ as $t\rightarrow \infty$ for all $x'\sim x$. By connectedness of $X$ and the triangle inequality, we also have 
$$ \lim_{t \to \infty} |P_tf(y)-P_tf(x)| = 0 \quad \text{for all $x,y \in X$}. $$
Letting $T\rightarrow \infty$, we conclude that 
\begin{align}
\diam(X,d_\Gamma) &= f(y_0)-f(x_0) \nonumber \\
&\le |f(x_0)-P_\infty f(x_0)|+|f(y_0)-P_\infty f(y_0)|+|P_\infty f(y_0)-P_\infty f(x_0)| \label{eq:Ptfest}\\ &\le \frac{2\sqrt{2c}}{K}. \nonumber
\end{align}
This finishes the proof of the diameter estimate for $d_\Gamma$. The diameter estimate for $d$ follows from the relation \eqref{eq:ddGamma}.
\end{proof}

\begin{remark}
The same arguments would lead to the diameter result in \cite{LMP-18} for Bakry-\'Emery curvature if one would have chosen $\varepsilon=0$ from the beginning onwards.
\end{remark}

\begin{remark} Further developments in the case of Bakry-\'Emery curvature allowing some negative curvature locally can be found in \cite{Mu-18,LMPR-19}.
\end{remark}

\begin{remark}
In the case of Bakry-\'Emery curvature, the diameter bound leads also to a homology vanishing result due to the local character of the curvature \cite{MR-20,KMY-21}. It is unlikely that a similar result can be derived for entropic curvature due to its non-local nature.
\end{remark}

\subsection{Diameter bound under $CD_\theta(K,n)$} \label{subsec:diambd-finite}

In this subsection we derive another diameter bound in the case of a finite dimension parameter $n$.

\begin{theorem} [Diameter bound under $CD_\theta(K,n)$]
Let $(X,Q,\pi)$ be a finite Markov chain and $\theta$ a mean satisfying $\theta \le \theta_a$.
Let $K>0$ and $n \in (0,\infty)$. If $X$ satisfies
    $CD_\theta(K,n)$, then
    \[
     {\mathrm{diam}}(X,d_\Gamma) \leq  \pi \sqrt{\frac n K}.
    \]
In particular, we have
    \[
     {\mathrm{diam}}(X,d) \leq  \pi \sqrt{\frac{Dn} {2K}},
    \]
where $D := \max_{x \in X} D(x)$ is the maximal weighted vertex degree with $D(x)$ given in \eqref{eq:weighdeg}.
\end{theorem}

\begin{proof}
    Let $x_0,y_0 \in X$ be such that $d_\Gamma(x_0,y_0) = {\rm diam}(X,d_\Gamma)$ and $f \in \mathbb R^X$ such that $\Gamma f \leq 1$ and $f(y_0)-f(x_0) = {\rm diam}(X,d_\Gamma)$.
    We aim to show $|f(y)-f(x)| \leq \pi \sqrt{\frac n K}$.

    By \eqref{eq:Agammaest} in Lemma \ref{lem:Arhoprop}
    and $\theta \le \theta_a$,
    we have for all $\rho \in \mathbb (0,\infty)^ X$,
    \[
    \mathcal A_\rho (f) \leq \langle \rho, \Gamma f \rangle_\pi \leq \|\rho\|_1 := \sum_{x \in X} |\rho(x)| \pi(x).
    \]
    By the gradient estimate in Theorem \ref{thm:gradest},
    \[
    \langle\rho, (\Delta P_t f)^2 \rangle_\pi \leq \frac{Kne^{-2Kt}}{1-e^{-2Kt}}\mathcal A_{P_t \rho} f \leq \frac{Kne^{-2Kt}}{1-e^{-2Kt}}\, \|\rho\|_1.   
    \]
    This inequality extends obviously to all measures $\rho \in [0,\infty)^X$ and, letting
    $\rho$ be supported at a single vertex $x \in X$, we obtain
    \[
    (\Delta P_t f)^2(x) \leq \frac{Kne^{-2Kt}}{1-e^{-2Kt}}.
    \]
Integrating over $t$ leads to
\[
|P_\infty f(x) - f(x)| \leq \int_0^\infty \sqrt{\frac{Kne^{-2Kt}}{1-e^{-2Kt}}}\, dt = \frac \pi 2 \sqrt{\frac n K}.
\]
As $P_\infty f$ is constant on $X$, the first diameter estimate follows again via the estimate \eqref{eq:Ptfest} from the previous proof. For the second diameter estimate we employ again the relation \eqref{eq:ddGamma}.
\end{proof}

\section{No expanders with non-negative entropic curvature}

In this section, we follow the arguments of an unpublished preliminary version of \cite{MS-23}
for Ollivier Ricci curvature. These arguments can be adapted to show that there is no expander family of graphs with non-negative entropic curvature. In particular, 
we derive an upper bound on the smallest positive eigenvalue $\lambda_1(X)$ of $-\Delta$ on a finite Markov chain with non-negative entropic  curvature in terms of $\pi_{\min}, \pi_{\max}$ and $Q_{\min}$ alone (see Corollary \ref{cor:lambda_uppbd} below).
This will imply for increasing $d$-regular combinatorial graphs, which are non-negatively curved with respect to the simple random walk, that their spectral gaps tend to zero.  Note that the curvature notion in \cite{MS-23}, namely Ollivier Ricci curvature, is different from the entropic curvature used in this paper. Despite this fact, the non-negative curvature in the sense of either curvature notion implies the same reverse Poincar\'e-type estimate (see Corollary \ref{cor:rp_zeroinf}), that is,
\begin{equation*}
\|\nabla P_tf\|_\infty \le \frac{1}{\sqrt{t}} \frac{\|f\|_\infty}{\sqrt{Q_{\min}}}.
\end{equation*}
This gradient estimate is the key ingredient for proving an upper bound on the product of $\lambda_1(X)$ and the average mixing time, which only involves $Q_{\min}$ (see Theorem \ref{thm:lambda-tau}). 

\subsection{Cheeger constant and average mixing time}
\label{subsec:cheegavmix}

Let $(X,Q,\pi)$ be a finite Markov chain. 
For $W \subseteq X$, let
\[
|\partial W| := \sum_{\substack{x \in W\\y\notin W}}w(x,y)
\]
and $\pi(W):=\sum_{x \in W} \pi(x)$.
The \emph{Cheeger constant} of $X$ is defined as
\[
h(X):=\inf_{\pi(W) \leq \frac{1}{2}} \frac{|\partial W|}{\pi(W)}.
\]
For $f \in \R^X$, we define the $\ell_1$-norms:
\begin{align*}
    \|f\|_1 &:= \sum_{x \in X} |f(x)|\pi(x), \\
    \|\nabla f\|_1 &:= \frac 1 2 \sum_{x,y \in X} |f(y)-f(x)|Q(x,y)\pi(x) = \frac 1 2 \sum_{x,y \in X} |f(y)-f(x)| w(x,y).
\end{align*}

The Cheeger constant can be seen as some $\ell_1$-Rayleigh quotient as stated in the next lemma.
\begin{lemma}\label{lem:CheegerL1Gradient}
Suppose a function $f \in \R^X$ satisfies $\sum_x f(x) \pi(x)=0$. Then
\[
\|\nabla f\|_1 \geq \frac{h(X)}{2}  \|f\|_1.
\]

\end{lemma}
\begin{proof}
Without loss of generality, we assume $\pi(\{x \in X: f(x) >0\}) \leq 1/2$. Define a function $f_{+}$ as $f_{+}(x):= \max \{ f(x), 0\}$. We notice $\frac 1 2\|f\|_1 = \|f_+\|_1$ by the assumption of the lemma. Setting $\int_a^b dt=0$ if $a>b$ and defining the level set $F_t := \{x\in X :f(x) > t\}$  for $t\in \mathbb{R}$, we obtain
\begin{align*}
\|\nabla f\|_1 \geq \|\nabla f_+ \|_1 &=\sum_{x\sim y} w(x,y)\int_{f_+(x)}^{f_+(y)} dt = \int_{0}^\infty |\partial F_t| dt \\
&\geq  h(X)\int_{0}^\infty \pi(F_t) dt = h(X)\|f_+\|_1 = \frac{h(X)}{2} \|f\|_1, 
\end{align*}
where we used in the estimate that $\pi(F_t) \leq \pi(F_0) \leq 1/2$ for $t \ge 0$.
This finishes the proof.
\end{proof}

The \emph{average mixing time} is defined as
\begin{equation} \label{eq:avmixt}
\tau_{\rm avg}(\varepsilon) := \inf \left\{t>0: \sum_{x,y \in X} \pi(x)\pi(y)\left|p_t(x,y) - 1\right| \le \varepsilon\right\},
\end{equation}
where $p_t(x,y)$ is the heat kernel defined by $$ p_t(x,y):=\frac{1}{\pi(x)}P_t\mathbf{1}_x(y) = P_t \delta_x(y) $$
with $\delta_x$ defined in \eqref{eq:deltaxcurv}. This implies $p_t(x,y) = p_t(y,x)$ and
$$ P_t f(x) = \sum_{y \in X} p_t(x,y)f(y) \pi(y). $$

\begin{remark}
  Note that in the definition of the average mixing time \eqref{eq:avmixt}, the function
  $$ t \mapsto \sum_{x,y \in X} \pi(x)\pi(y)|p_t(x,y)-1| $$
  is monotone decreasing, since $P_t$ is a family of $L^\infty$-contractions and therefore also of $L^1$-contractions (see \cite[Theorem 1.3.3]{Dav-89}). 
\end{remark}

\begin{lemma} \label{lem:heatkernel_uppbd}
Let $r$ be any non-negative integer. For all $t>0$ and all $x,y\in V$ such that $d(x,y) \ge r$, the heat kernel $p_t(x,y)$ satisfies
\begin{equation} \label{eq:heatkernel_uppbd}
    p_t(x,y) \le \frac{1}{\pi(x)}\frac{t^r}{r!}.
\end{equation}
\end{lemma}
\begin{proof}
    We will prove $P_t\mathbf{1}_x(y) \le \frac{t^r}{r!}$ by induction on $r$. This inequality is obvious for the base case $r=0$. For the induction step, consider $x,y$ such that $d(x,y) \ge r+1$. Note that $d(x,z) \ge r$ for all $z\sim y$, and hence $P_t\mathbf{1}_x(z) \le \frac{t^r}{r!}$ by the induction hypothesis. It follows that
    \begin{align*}
        \partial_t P_t \mathbf{1}_x (y) &= \Delta (P_t \mathbf{1}_x)(y) = \sum_{z:z\sim y} Q(y,z) (P_t \mathbf{1}_x(z)-P_t \mathbf{1}_x(y)) \\ &\le \sum_{z:z\sim y} Q(y,z)\left(\frac{t^r}{r!} \right)\le \frac{t^r}{r!}.
    \end{align*}
    Integrating from $t=0$ to $t=\tau$, we obtain $P_\tau \mathbf{1}_x (y) \le \frac{\tau^{r+1}}{(r+1)!}$ as desired.
\end{proof}

For a given finite Markov chain $(X,Q,\pi)$ we use the notation $\pi_{\min} := \min_{x \in X} \pi(x)$, $\pi_{\max} := \max_{x \in X} \pi(x)$ and $Q_{\min} := \min_{x \sim y} Q(x,y)$. The following Lemma assumes $Q_{\min} < 1$ (that is, $X$ is not $K_2$ with the simple random walk), and it provides a lower bound of the average mixing time.

\begin{lemma} \label{lem:tau}
Let $(X,Q,\pi)$ be a finite Markov chain with $\pi_{\max} <\frac{1}{4}$, $Q_{\min} < 1$
and $R_0:=\frac{\log(4\pi_{\max})}{\log(Q_{\min})}>0$. Then the average mixing time satisfies
\[\tau_{\rm avg}(1/4) \ge \left(\frac{ \pi_{\min}}{8\pi_{\max}} \right)^{\frac{1}{R_0}} \frac{Q_{\min}}{e}R_0.\]
\end{lemma}

\begin{proof}
We first observe that $(Q_{\min})^{R_0}=4\pi_{\max}$. The assumptions on $\pi_{\max}$ and $Q_{\min}$ imply that $R_0 > 0$.
Denoting 
$T:= \left(\frac{ \pi_{\min}}{8\pi_{\max}} \right)^{1/R_0} \frac{Q_{\min}}{e}R_0$, we need to show that $\tau_{\rm avg}(1/4) \ge T$. By Lemma \ref{lem:heatkernel_uppbd} and the fact that $n! > \left( \frac{n}{e} \right)^n$ for all $n \in \mathbb{N}$, we have for any $x,y\in X$ with $d(x,y)=R\ge R_0>0$ and for any $t\in [0,T]$,
\begin{align*}
0 \le p_t(x,y) \le \frac{1}{\pi(x)}\frac{t^R}{R!} \le \frac{1}{\pi(x)}\left(\frac{eT}R\right)^{R} 
&\le \frac{1}{\pi(x)} \left(\frac{\pi_{\min}}{8\pi_{\max}} \right)^{\frac{R}{R_0}} Q_{\min}^R \\
&\le \frac{1}{\pi(x)} \left(\frac{\pi_{\min}}{8\pi_{\max}} \right) 4\pi_{\max}\le \frac{1}{2}.\end{align*}
On the other hand, 
if $d_{\max} := \max_{x \in X} d_x$ denotes the maximal combinatorial vertex degree of the underlying graph of the Markov chain, then we have $d_{\max} \le 1/Q_{\min}$ and $|B_r(x)| \le 2(d_{\max})^r$ for any radius $r \ge 0$.
In particular, we have
\[\pi(B_{R_0}(x)) \le \pi_{\max}|B_{R_0}(x)| \le \pi_{\max}\cdot 2(d_{\max})^{R_0}
\le 2\pi_{\max}Q_{\min}^{-R_0} = \frac{1}{2}.\]
It follows from the two inequalities above that for all $t\in [0,T]$,
\begin{align*} 
\sum_{x,y \in V} \pi(x)\pi(y)\left|p_t(x,y) - 1\right|
&\ge \sum_x \pi(x) \left(\frac{1}{2} \sum_{y: d(x,y)> R_0}\pi(y) \right) \\
&\ge \frac{1}{4} \sum_x \pi(x) = \frac{1}{4},
\end{align*}
which then implies $\tau_{\rm avg}(1/4) \ge T$, as desired.
\end{proof}

\subsection{Spectral gap and non-negative curvature}
\label{subsec:specgapnonegcurv}

Recall that $Q_{\min} = \min_{x \sim y} Q(x,y)$. We start with the following Buser inequality.

\begin{theorem}[Buser inequality] \label{thm:Buser}
    Let $(X,Q,\pi)$ be a finite Markov chain satisfying $CD_{ent}(0,\infty)$. Then we have the following Buser inequality:
    \begin{equation} \label{eq:Buser}
        \lambda_1(X) \le \frac{16\log(2)}{Q_{\min}}h(X)^2,
    \end{equation}
    where $\lambda_1(X)$ is the smallest positive eigenvalue of $- \Delta$ and $h(X)$ is the Cheeger constant from the previous subsection.
\end{theorem}

\begin{proof}
    This Buser inequality is proved in \cite[Theorem 3.2.2]{Mu-23} in the case of non-negative Ollivier curvature. However, the author only uses the non-negative curvature assumption in the form of $\|P_t\mu-P_t\nu\|_1 \le \frac{W_1(\mu,\nu)}{\sqrt{tQ_{\min}}}$ (see Corollary 3.2.1 and the proof of Lemma 3.2.1 in \cite{Mu-23}). 
    The same inequality is also achieved under the non-negative entropic curvature assumption (see Corollary \ref{cor:rp_zeroinf}), and therefore, the same proof of Buser inequality in \cite{Mu-23} follows verbatim.

    We remark that, under the non-negative entropic curvature assumption, a slightly different Buser inequality is proved in \cite[Corollary 4.2]{EF-18}: 
    $$\tilde h(X) = \inf_{\emptyset \neq W \subsetneqq X} \frac{|W|}{\pi(W)(1-\pi(W))} \ge \frac{\sqrt{\lambda_1(X) Q_{\min}}}{2}$$ 
    with a different notion of Cheeger constant $\tilde h(X)$. Using the fact that $\tilde h(X) \le 2h(X)$, we may derive $\lambda_1(X) \le \frac{16}{Q_{\min}}h^2(X)$, which has a slightly worse constant than \eqref{eq:Buser}.
\end{proof}

\begin{theorem} \label{thm:lambda-tau}
Let $(X,Q,\pi)$ be a finite Markov chain satisfying $CD_{ent}(0,\infty)$. Then we have
\[
    \lambda_1(X) \tau_{\rm avg}(1/4) \le \frac{256\log(2)}{(Q_{\min})^2}.
\]
\end{theorem}

\begin{proof}
    By Corollary \ref{cor:rp_zeroinf}, we have for any pair $x \sim y$ of adjacent vertices,
    $$ \Vert p_t(y,\cdot) - p_t(x,\cdot) \Vert_1 = \Vert P_t \delta_y - P_t \delta_x \Vert_1 \le \frac{W_1(\delta_y,\delta_x)}{\sqrt{t Q_{\min}}} = \frac{1}{\sqrt{t Q_{\min}}}, $$
    which implies
    $$ \frac{1}{2} \sum_{x,y} w(x,y) \Vert p_t(y,\cdot) - p_t(x,\cdot) \Vert_1 \le \frac{1}{\sqrt{t Q_{\min}}} \cdot \frac{1}{2} \sum_{x,y} w(x,y) = \frac{1}{2 \sqrt{t Q_{\min}}}. $$
    
    On the other hand, by rearranging the sums and applying the Cheeger $\ell_1$-gradient estimate (Lemma \ref{lem:CheegerL1Gradient} with $f=p_t(z,\cdot)-1$), we obtain
    \begin{align*}
        \frac{1}{2} \sum_{x,y} w(x,y)\|p_t (y,\cdot)-p_t (x,\cdot)\|_1 &= \frac{1}{2}\sum_{x,y,z} w(x,y)\pi(z)|p_t(y,z)-p_t(x,z)|\\
        &= \sum_{z} \pi(z)\|\nabla p_t(z,\cdot)\|_1\\
        &\ge \frac{h(X)}{2} \sum_{z} \pi(z)\|p_t(z,\cdot) -1 \|_1 \\
        &=\frac{h(X)}{2} \sum_{y,z} \pi(z)\pi(y)|p_t(z,y) -1|,
    \end{align*}
    where the last expression is at least $\frac{h(X)}{8}$ if $t \le \tau_{\rm avg}(1/4)$. Thus
    \[ \frac{1}{2\sqrt{\tau_{\rm avg}(1/4) Q_{\min}}} \ge \frac{h(X)}{8},\]
    that is,
    \[ h(X)^2 \le \frac{16}{\tau_{\rm avg}(1/4)Q_{\min}}.\]
    Together with the Buser inequality \eqref{eq:Buser}, we obtain
    \[
        \lambda_1(X) \le \frac{16\log(2)h(X)^2}{Q_{\min}} \le \frac{256\log(2)}{\tau_{\rm avg}(1/4)(Q_{\min})^2}.
    \]
\end{proof}

The following corollary is an immediate consequence of Lemma \ref{lem:tau} and Theorem \ref{thm:lambda-tau}.

\begin{corollary} \label{cor:lambda_uppbd}
    Let $(X,Q,\pi)$ be a finite Markov chain satisfying $CD_{ent}(0,\infty)$. If $\pi_{\max}<1/4$ and $Q_{\min} <1$, we have
    \begin{equation} \label{eq:lambda_uppbd}
\lambda_1(X) \le \frac{483}{(Q_{\min})^3} \left(\frac{8\pi_{\max}}{\pi_{\min}} \right)^{\frac{1}{R_0}} \frac{1}{R_0},
    \end{equation}
    where $R_0:=\frac{\log(4\pi_{\max})}{\log(Q_{\min})}$.
\end{corollary}

Finally, we present the following combinatorial result for expander graphs.

\begin{corollary}
    Let $G=(X,E)$ be a finite connected $d$-regular combinatorial graph with vertex set $X$ and edge set $E$. Assume $d \ge 2$ and that the associated simple random walk Markov chain has non-negative entropic curvature (more precisely, $(X,Q,\pi)$ with $Q(x,y) = 1/d$ for $y \sim x$ and $\pi(x) = 1/|X|$ satisfies $CD_{ent}(0,\infty)$).
    Then the spectral gap of $G$ with $|X| \ge 4 d$ satisfies
    \[
        d-\mu_2(G) \le \frac{4000\, d^4 \log d}{\log(|X|/4)},
    \]
    where $\mu_2(G)$ is the second largest eigenvalue of the adjacency matrix $A_G$ of $G$.
    
    In particular, we have $d-\mu_2(G_k) \to 0$ for families of connected $d$-regular graphs $G_k=(X_k,E_k)$ with $|X_k| \to \infty$, that is, there are no expanders of non-negative entropic curvature.
\end{corollary}

\begin{proof}
    In the case of a $d$-regular graph,
    we note that $- \Delta = {\rm{Id}}_X - \frac{1}{d} A_G$, $Q_{\min}=1/d$ and $\pi_{\min}=\pi_{\max} = \frac{1}{|X|}$. Then the spectral gap inequality follows straightforwardly from \eqref{eq:lambda_uppbd}.
\end{proof}

{\bf{Acknowledgements:}} We thank Matthias Erbar for stimulating discussions. Shiping Liu is supported by the National Key R and D Program of China 2020YFA0713100 and the National Natural Science Foundation of China No. 12031017. Supanat Kamtue is supported
by Shuimu Scholar Program of Tsinghua University No. 2022660219.

\bibliographystyle{apalike}
\bibliography{Bibliography}

\end{document}